\UseRawInputEncoding
\documentclass[12pt,reqno]{article}

\usepackage{amsfonts}
\usepackage{mathrsfs}

\usepackage{amsmath,amsthm,amsfonts,amssymb,latexsym,mathrsfs,color,bm}
\usepackage{float}

\usepackage[colorlinks=true,
linkcolor=webblue, filecolor=webbrown, citecolor=webred ]{hyperref}
\definecolor{webblue}{rgb}{0,.5,0}
\definecolor{webred}{rgb}{0,.5,0}
\definecolor{webbrown}{rgb}{.6,0,0}

\newtheorem{thm}{Theorem}[section]
\newtheorem{lem}[thm]{Lemma}
\newtheorem{coro}[thm]{Corollary}
\newtheorem{prop}[thm]{Proposition}
\newtheorem{conj}[thm]{Conjecture}
\newtheorem{re}[thm]{Remark}

\newtheorem{cl}{Claim}

\theoremstyle{definition}
\newtheorem{definition}[thm]{Definition}
\newtheorem{ex}[thm]{Example}

\newtheorem{rem}[thm]{Remark}

\numberwithin{equation}{section}

\newcommand{\D}{\displaystyle}
\newcommand{\DF}[2]{\frac{\D#1}{\D#2}}

\title{Stability of combinatorial polynomials and its applications
\thanks{Supported partially by the National Natural Science Foundation of China (Nos. 11971206, 12022105) and the
Natural Science Fund for Distinguished Young Scholars of Jiangsu Province (No. BK20200048).
\newline\hspace*{5mm}
{\it Email address:} ding-mj@hotmail.com (M.-J. Ding),
bxzhu@jsnu.edu.cn (B.-X. Zhu)}}
\author{Ming-Jian Ding$^a$ and Bao-Xuan Zhu$^b$}
\date{\footnotesize
$^a$School of Mathematical Sciences, Dalian University of Technology, Dalian 116024, PR China\\
$^b$School of Mathematics and Statistics, Jiangsu Normal University,
Xuzhou 221116, PR China}

\begin{document}

\maketitle

\begin{abstract}
Many important problems are closely related to the zeros of
certain polynomials derived from combinatorial objects. The aim
of this paper is to make a systematical study on the stability of
polynomials in combinatorics.

Applying the characterizations of Borcea and Br\"and\'en concerning
linear operators preserving stability, we present criteria for real
stability and Hurwitz stability of recursive polynomials. We also
give a criterion for Hurwitz stability of the Tur\'{a}n expressions
of recursive polynomials. As applications of these criteria, we
derive some stability results occurred in the literature in a unified
manner. In addition, we obtain the Hurwitz stability of Tur\'{a}n
expressions for alternating runs polynomials of types $A$ and $B$
and solve a conjecture concerning Hurwitz stability of alternating
runs polynomials defined on a dual set of Stirling permutations.

Furthermore, we prove that the Hurwitz stability of any symmetric
polynomial implies its semi-$\gamma$-positivity. We study a class of
symmetric polynomials and derive many nice properties including
Hurwitz stability, semi-$\gamma$-positivity, non $\gamma$-positivity,
unimodality, strong $q$-log-convexity, the Jacobi continued fraction
expansion and the relation with derivative polynomials. In particular,
these properties of the alternating descents polynomials of types
$A$ and $B$ can be obtained in a unified approach.

Finally, based on the $h$-polynomials from combinatorial
geometry, we use real stability to prove a criterion for zeros
interlacing between a polynomial and its reciprocal polynomial,
which in particular implies the alternatingly increasing property
of the original polynomial. This criterion extends a result of
Br\"and\'en and Solus and unifies such properties for many
combinatorial polynomials, including ascent polynomials for $k$-ary
words, descent polynomials on signed Stirling permutations and
colored permutations and $q$-analog of descent polynomials on
colored permutations, and so on. Furthermore, we also obtain
a recurrence relation and zeros interlacing of $q$-analog of
descent polynomials on colored permutations that extend some
results of Br\"and\'en and Brenti. In addition, as an application
of Hurwitz stability, we prove the alternatingly increasing property
and zeros interlacing for two kinds of peak polynomials on the dual
set of Stirling permutations.
\end{abstract}
\tableofcontents

\newpage

\section{Introduction}
The analytic theory of polynomials plays a significant role in
different fields, such as analysis, combinatorics, probability,
optimization, real algebraic geometry, automatic control theory and
statistical physics, see the monograph \cite{M66}. In particular,
the theory of multivariate stable polynomials recently displays more
and more power to solve some hard problems
\cite{BB09,BB09CPAM,BB10,BBT09,Bra07,Wan211}. The problems center in
the analytic theory of polynomials is the study of the zeros or
coefficients. The zeros of a polynomial can often reveal a variety
of information. In addition, many important problems can be
transformed to the distribution of zeros of polynomials, such as the
four color problem \cite{Bir12}, the Riemann hypothesis
\cite{GNRZ19}, the Lee-Yang program on phase transitions in
equilibrium statistical mechanics \cite{LY52,YL52}, and the
construction of Ramanujan graphs \cite{MMS15}. In combinatorics, the
zeros of polynomials are often used to determine the (combinatorial)
information of the coefficients, such as asymptotical normality,
unimodality, log-concavity, $q$-log-convexity, $\gamma$-positivity,
P\'olya frequency, total positivity, alternatingly increasing
property, see \cite{Bra15,Bre94,Sta89}.

The differential operators often arise in analysis. Many classical
orthogonal polynomials can be generated from different differential
operators, such as Legendre polynomials
$L_n(x)=\frac{1}{2^nn!}D_x^n(x^2-1)^n$, Laguerre polynomials
$\mathcal {L}_n(x)=e^{x}D_x^n(x^ne^{-x})$, Hermite polynomials
$H_n(x)=(-1)^ne^{x^2}D_x^ne^{-x^2}$, where $D_x=d/dx$. In addition,
orthogonal polynomials often satisfy certain differential
recursive relations, for example, the Jacobi polynomial
$P^{(\alpha,\beta)}_n(x)$ satisfies
$$
   2nP^{(\alpha,\beta)}_n(x)
 = [\alpha-\beta+x(\alpha+\beta+2)]P^{(\alpha+1,\beta+1)}_{n-1}(x)
   -(x^2-1)D_xP^{(\alpha+1,\beta+1)}_{n-1}(x).
$$
The combinatorial polynomials often also have such property. For
example,
$$
(xD_x)^n\left(\frac{1}{1-x}\right)=\frac{xA_n(x)}{(1-x)^{n+1}},
$$
where $A_n(x)$ is the classical Eulerian polynomial. We refer the
reader to \cite{Ag14, BF10, CS88} for more combinatorial polynomials
generated in this way. On the other hand, the classical Eulerian
polynomial $A_n(x)$ satisfies the recurrence relation
\begin{equation}\label{Rec+Euler+A}
A_n(x) = [(n-1)x+1]A_{n-1}(x) + x(1-x)D_xA_{n-1}(x),
\end{equation}
where $A_0(x)=1$. In fact, for some combinatorial sequences, their
recurrence relations are very nice feature, which are a useful way to
study many properties. In this paper, we will mainly consider the
zeros distribution of the polynomial $T_n(x)$ satisfying the
following generalized recurrence relation:
\begin{equation}\label{Rec+main+intro}
  T_{n+1}(x) = (\alpha_nx^2+\beta_{n}x+\gamma_{n})T_{n}(x)
  + (\mu_{n}x^3+\nu_{n}x^2+\varphi_{n}x+\psi_n)D_xT_{n}(x),
\end{equation}
where all $\alpha_n,\beta_{n},\gamma_{n}, \mu_{n},\nu_{n},
\varphi_{n},\psi_n$ are real sequences in $\mathbb{R}$.

In Section $2$, with the help of the characterizations of Borcea and
Br\"and\'en concerning the linear operator preserving stability
\cite{BB09}, we present criteria for the real stability of $T_n(x)$
(see Theorem \ref{thm+RS}) and the Hurwitz stability of $T_n(x)$ for
$\nu_{n}=\psi_n=0$ (see Theorem \ref{thm+HS}). These criteria can
be applied to a large number of combinatorial polynomials, such as the
generalized Eulerian polynomials, the Stirling-Whitney-Riordan
polynomials, and deal with those known results occurred in
the literature \cite{HWY15, WY06, YZ15, Zhu20jcta, Zhu21} in a
unified approach. In particular, we obtain the Hurwitz stability of
alternating runs polynomials defined on a dual set of Stirling
permutations, which solves a conjecture in \cite{MW16}. Furthermore,
we give a criterion for Hurwitz stability of certain linear
combination of $T_n(x)$ for $\alpha_n=\mu_{n}=\psi_n=0$ (see
Proposition \ref{prop+HS+linear+comb}), which extends a
corresponding result for $A_n(x)$ due to Zhang and Yang \cite{YZ15}.

In Section $3$, we prove a result for the Hurwitz stability of
a nonlinear operator on polynomials called the Tur\'{a}n expression
 (see Theorem \ref{thm+Turan}). It unifies plenty of known
results in \cite{CGV15,Z18ejc,Zhu20jcta,Zhu21}. In addition, we also
prove the Hurwitz stability of Tur\'{a}n expressions for alternating
runs polynomials of types $A$ and $B$, up-down runs polynomials, and
so on. In particular, Hurwitz stability of these Tur\'an expressions
implies $q$-log-convexity of the original polynomial sequence,
repectively.

The symmetric polynomials often have more nice properties. In
Section $4$, we prove that Hurwitz stability of any symmetric
polynomial implies its semi-$\gamma$-positivity (see Theorem
\ref{thm+sem+gamma}), which is similar to that real rootedness of
any symmetric polynomial implies the $\gamma$-positivity (see
Br\"and\'en \cite{Bra04}). Moreover, we demonstrate the Hurwitz
stability and semi-$\gamma$-positivity for a class of symmetric
polynomials $T_n(x)$ for $\alpha_n=-m_n\mu_n$, $\gamma_n=\beta_n
+m_n\nu_n$, $\varphi_{n}=-\nu_{n}$ and $\psi_n=-\mu_{n}$, where
$m_n=\deg(T_n(x))$ (see Theorem \ref{thm+sym+HS}). We also derive
many other nice properties including unimodality, non
$\gamma$-positivity, strong $q$-log-convexity, the Jacobi
continued fraction expansion and the relation with derivative
polynomials. In particular, these properties of the alternating
descents polynomials of types $A$ and $B$ can be obtained in a
unified approach.

In Section $5$, based on the $h$-polynomials from combinatorial
geometry, we present a criterion for zeros interlacing between a
polynomial and its reciprocal polynomial, which in particular
implies the alternatingly increasing property of the original
polynomial (see Theorem \ref{thm+alter+incr}). This criterion
extends a result of Br\"and\'en and Solus \cite{BS21} and unifies
such properties for many combinatorial polynomials, including ascent
polynomials for $k$-ary words, descent polynomials on signed
Stirling permutations and colored permutations and $q$-analog of
descent polynomials on colored permutations, and so on. On the other
hand, we obtain a recurrence relation and zeros interlacing of
$q$-analog of descent polynomials on colored permutations that
extend some results of Br\"and\'en \cite{Bra06} and Brenti
\cite{Bre94EJC}. Finally, using our results for Hurwitz stability,
we show the alternatingly increasing property and zeros interlacing
for two kinds of peak polynomials on the dual set of Stirling
permutations.

The next is the definition of some notations. Denote $\mathbb{N}^+,
\mathbb{N}, \mathbb{R}^{>0}, \mathbb{R}^{\ge0}, \mathbb{R}$ and
$\mathbb{C}$ be the positive integers, nonnegative integers, positive
real numbers, nonnegative real numbers, real numbers and complex numbers,
respectively. Let $\mathbb{R}[x]$ (resp., $\mathbb{C}[x]$) denote the
set of polynomials over $\mathbb{R}$ (resp., $\mathbb{C}$) and
$\mathbb{R}_n[x]$ (resp., $\mathbb{C}_n[x]$) denote the set of polynomials
with degree at most $n$ over $\mathbb{R}$ (resp., $\mathbb{C}$). Let
$S_n$ represent the symmetric group on $[n] =\{1, 2, \dots, n\}$.

\section{Stability of polynomials}

\subsection{Definitions of stability}
Let $H\subset\mathbb{C}$ be an open half-plane whose
boundary contains the origin, namely $H=\left\{ z \in \mathbb{C} |\,
\Im(e^{i\theta}z) > 0 \right\}$ for $\theta \in \mathbb{R}$, where
$\Im(z)$ is the image part of $z$ for $z \in \mathbb{C}$. We say that
$f \in \mathbb{C}[z_1, \cdots, z_n]$ is {\it $H$-stable} if it is
either identically zero or nonvanishing whenever $z_i \in H$ for any $i
\in [n]$. In particular, $f$ is called {\it stable} if $H$ is the upper
half-plane ($\theta = 0$), and $f$ is {\it real stable} if all
coefficients of $f$ are real. Clearly, a univariate polynomial $f$
is real stable if and only if $f$ has only real zeros. Similarly,
$f$ is called {\it Hurwitz stable} if $H$ is the right half-plane
($\theta = \pi/2$). We will consider the real stability and
Hurwitz stability of the polynomials in this paper.

Let $f, g \in \mathbb{R}\left[x \right]$ be real-rooted
with zeros $\left\{r_i\right\}$ and $\left\{s_j\right\}$, respectively.
We say that $g$ \emph{interlaces} $f$ if $\deg(f)=\deg(g)+1=n$ and
\begin{equation}\label{1}
 r_{n} \leq s_{n-1}\leq \cdots  \leq s_{2} \leq r_{2} \leq s_{1} \leq r_{1},
\end{equation}
and that $g$ \emph{alternates left of} $f$ if $\deg(f)=\deg(g)=n$ and
\begin{equation}\label{2}
 s_{n} \leq r_{n} \leq \cdots  \leq s_{2} \leq r_{2} \leq s_{1} \leq r_{1}.
\end{equation}

Denote either $g$ \emph{interlaces} $f$ or $g$ \emph{alternates left of}
$f$ by $g \preceq f$. If no equality sign occurs in \eqref{1} and \eqref{2},
then we say that $g$ \emph{strictly interlaces} $f$ and $g$
\emph{strictly alternates left of} $f$, respectively, denoted $g \prec f$.
Here, we denote $g \ll f$ if $g \preceq f$ and the leading coefficients of
$f, g$ have same sign or $f \preceq g$ and the leading coefficients of
$f, g$ have opposite sign. The following Hermite-Biehler Theorem
(see \cite[Theorem 6.3.4]{RS02}), which is a very classical result in
geometry of polynomials, characterizes two zeros-interlacing polynomials.

\begin{thm}[Hermite-Biehler Theorem]\label{thm+HB}
Let $\{f(x), g(x)\} \subseteq \mathbb{R}[x]$. Then $g(x) \ll f(x)$
if and only if $f(x)+ig(x)$ is stable.
\end{thm}

Following Theorem \ref{thm+HB}, we state an important result
obtained by Borcea and Br\"and\'en as follows.

\begin{prop}\label{prop+covn}\emph{\cite[Lemma 2.6]{BB10}}
Let $f(x)$ be a real-rooted polynomial that is not identically zero.
The sets
$$
\{g(x)\in \mathbb{R}[x]: g(x) \ll f(x)\} \quad \text{and} \quad
\{g(x)\in \mathbb{R}[x]: f(x) \ll g(x)\}
$$
are convex cones.
\end{prop}

In addition, for Theorem \ref{thm+HB}, Borcea and Br\"and\'en
\cite{BB09} gave an equivalent result: For $f(x), g(x) \in
\mathbb{R}[x]$, the stability of $f(x)+ig(x)$ is equivalent to that
of the bivariate polynomial $f(x)+yg(x)$. Thus, in order to show the
alternating property of zeros of two polynomials, the real stability
of bivariate polynomials is very useful.

For a linear operator
$
\mathbb{T} : \mathbb{R}_n [z] \rightarrow \mathbb{R}[z],
$
we define its {\it algebraic symbol} in $\mathbb{R}[z,w]$ by
$$
    G_{\mathbb{T}}(z+w)
 := \mathbb{T}[(z+w)^n]
  = \sum\limits_{k \le n} \binom{n}{k} \mathbb{T}(z^{k}) w^{n-k}.
$$
The following result for linear operators preserving real
stability of multivariate polynomials is a powerful tool to study real
stability.

\begin{thm}\label{thm+RS+Peter}\emph{\cite[Theorem 2.2]{BB09}}
For $n \in \mathbb{N}$, let $\mathbb{T} : \mathbb{R}_n [z] \rightarrow
\mathbb{R}[z]$ be a linear operator. Then $\mathbb{T}$ preserves stability
if and only if either
\begin{itemize}
\item [\rm (a)] $\mathbb{T}$ has range of dimension at most two and is of
                the form
                $$\mathbb{T} (f) = \alpha(f)P +\beta(f)Q,$$
                where $\alpha, \beta$: $\mathbb{R}_n [z]\rightarrow \mathbb{R}$
                are linear functional and $P, Q$ are real stable polynomial
                such that $P \ll Q$, or
\item [\rm (b)] the bivariate polynomial $G_{\mathbb{T}}(z+w)$ is stable, or
\item [\rm (c)] the bivariate polynomial $G_{\mathbb{T}}(z-w)$ is stable.
\end{itemize}
\end{thm}

\subsection{Real stability}
For the recurrence relation (\ref{Rec+main+intro}), for brevity, let
$\bm{\alpha}$ (resp., $\bm{\beta},\bm{\gamma}, \bm{\mu},\bm{\nu},\bm{\varphi},
\bm{\psi})$ denote $\alpha_n$ (resp., $ \beta_{n},\gamma_{n}, \mu_{n},\nu_{n},
\varphi_{n},\psi_n)$. Then we can rewrite (\ref{Rec+main+intro}) as
\begin{equation}\label{Rec+main}
  T_{n+1}(x) = (\bm{\alpha} x^2+\bm{\beta} x+\bm{\gamma} )T_{n}(x)
  + (\bm{\mu} x^3+\bm{\nu} x^2+\bm{\varphi} x+\bm{\psi})D_xT_{n}(x).
\end{equation}
Let $\deg(T_n(x))= m_n$ and define $F(x)$ and $G(x)$ by
\begin{equation}\label{eq+FG}
\left\{
\begin{array}{lcc}
 F(x) = \bm{\alpha} x^2+\bm{\beta} x+\bm{\gamma},                                    \\
 G(x) = (\bm{\alpha}+m_n\bm{\mu}) x^3+(\bm{\beta}+m_n\bm{\nu}) x^2
        +(\bm{\gamma}+m_n\bm{\varphi})x+m_n\bm{\psi}.
\end{array}
\right.
\end{equation}

For the recurrence relation \eqref{Rec+main}, it can be generated
from a linear operator $T$ defined by
\begin{equation}\label{Ope+main}
T := (\bm{\alpha} x^2+\bm{\beta} x+\bm{\gamma} )I
     + (\bm{\mu} x^3+\bm{\nu} x^2+\bm{\varphi} x+\bm{\psi})D_x,
\end{equation}
where $I$ is the identity operator and $D_x$ is the differential
operator $d/dx$.
We present one of the main results concerning stability as follows.

\begin{thm}\label{thm+FG}
The operator $T$ defined by \eqref{Ope+main} preserves real stability if $F(x) \ll G(x)$.
\end{thm}

\begin{proof}

According to \eqref{Ope+main} and Theorem \ref{thm+RS+Peter}, it
suffices to show that $T(x+y)^{m_n}$ is real stable. Note that we have
\begin{eqnarray*}
  T(x+y)^{m_n}& = & (x+y)^{m_n-1}\left[(\bm{\alpha} x^2+\bm{\beta} x+\bm{\gamma})(x+y)
                    + m_n\left(\bm{\mu} x^3+\bm{\nu} x^2+\bm{\varphi} x+\bm{\psi} \right)\right]  \\
              & = & (x+y)^{m_n-1}[(\bm{\alpha}+m_n\bm{\mu})x^3+(\bm{\beta}+m_n\bm{\nu})x^2
                    +(\bm{\gamma}+m_n\bm{\varphi})x+m_n\bm{\psi}                                  \\
              &   & +(\bm{\alpha} x^2+\bm{\beta} x+\bm{\gamma})y]                                 \\
              & = & (x+y)^{m_n-1}\left[G(x)+F(x)y\right].
\end{eqnarray*}
Obviously, $(x+y)^{m_n-1}$ is real stable. Then we only need to show
that $G(x)+F(x)y$ is real stable. By Theorem \ref{thm+HB},
$G(x)+F(x)y$ is real stable if and only if $F(x) \ll G(x)$.
This completes the proof.
\end{proof}

Next, we will give the sufficient conditions for operator $T$ defined by
\eqref{Ope+main} preserving real stability according to the degree conditions
of $F(x)$ and $G(x)$.

\begin{thm}\label{thm+RS}
Assume that both the leading coefficients of $F(x)$ and $G(x)$ are
positive and $0\leq\deg(G(x))-\deg(F(x)) \leq  1$. If $T_{n_0}(x)$
is real stable, then so is $T_n(x)$ in \eqref{Rec+main} for
$n \ge n_0$ under any of the following conditions:
\begin{itemize}
\item [\rm (1)] $\deg(F(x))\leq 1$ and $\bm{\beta\gamma\varphi}
                -\bm{\gamma}^2\bm{\nu}-\bm{\beta}^2\bm{\psi} \ge 0$,
\item [\rm (2)] $\deg(F(x)) = \deg(G(x))  =  2$, $\bm{\psi}=0$ and
                $m_n(\bm{\beta}+m_n\bm{\nu})(\bm{\beta\varphi}-\bm{\gamma\nu})
                -\bm{\alpha}(\bm{\gamma}+m_n\bm{\varphi})^2 \ge 0$.
\end{itemize}
\end{thm}

\begin{proof}

We will prove that $T_n(x)$ is real stable by
induction on $n$. By the assumption, $T_{n_0}(x)$ is real stable. It
follows from  Theorem \ref{thm+FG} that $T_n(x)$ for $n \ge n_0$
is real stable if $F(x) \ll G(x)$. Thus, we will prove that both
conditions in (1) and (2) imply that $F(x) \ll G(x)$.

For (1), because $0\leq\deg(G(x))-\deg(F(x)) \leq  1$, we divide its
proof into the following three cases in terms of the degree
conditions.

Case $1$: $\deg(F(x))=0$ and $\deg(G(x)) \le 1$. Obviously, we have
$\bm{\alpha}=\bm{\bm{\beta}}=\bm{\nu}=0$. This implies
$\bm{\beta}\bm{\gamma}\bm{\varphi}-\bm{\gamma}^2\bm{\nu}-\bm{\beta}^2\bm{\psi}= 0$.
By the assumption that the leading coefficients of $F(x)$ and $G(x)$
are positive, we have $\bm{\gamma}>0$ and $\bm{\gamma}+m_n\bm{\varphi} > 0$.
Then the bivariate polynomial $G(x)+F(x)y$ is reduced to
$$
(\bm{\gamma}+m_n\bm{\varphi})x+m_n\bm{\psi}+\bm{\gamma} y,
$$
which is clearly real stable.

Case $2$: $\deg(F(x))=\deg(G(x))=1$. We have $\bm{\alpha}=\bm{\beta}+m_n\bm{\nu}=0$.
By the assumption that the leading coefficients of $F(x)$ and $G(x)$
are positive, we have $\bm{\beta}>0$ and $\bm{\gamma}+m_n\bm{\varphi} > 0$.
Thus the condition $\bm{\beta}\bm{\gamma}\bm{\varphi}-\bm{\gamma}^2\bm{\nu}
-\bm{\beta}^2\bm{\psi}\geq 0$ implies
\begin{equation}\label{Ineq+1}
\bm{\gamma}^2+m_n\bm{\gamma}\bm{\varphi}-m_n\bm{\beta}\bm{\psi} \ge 0.
\end{equation}
Then $F(x) \ll G(x)$ is reduced to
$$
\bm{\beta} x+\bm{\gamma} \ll (\bm{\gamma}+m_n\bm{\varphi})x+m_n\bm{\psi}.
$$
This is equivalent to
$$
-\frac{\bm{\gamma}}{\bm{\beta}} \leq -\frac{m_n\bm{\psi}}{\bm{\gamma}+m_n\bm{\varphi}} ,
$$
which follows from the inequality \eqref{Ineq+1}.

Case $3$: $\deg(F(x))=1$ and $\deg(G(x))=2$. By the assumption that
the leading coefficients of $F(x)$ and $G(x)$ are positive, we have
$\bm{\beta}>0$ and $\bm{\beta}+m_n\bm{\nu}
>0$. So $F(x) \ll G(x)$ is reduced to
$$
\bm{\beta} x+\bm{\gamma} \ll (\bm{\beta}+m_n\bm{\nu})x^2+(\bm{\gamma}+m_n\bm{\varphi})x+m_n\bm{\psi}.
$$
Obviously, the interlacing follows from
\begin{equation}\label{Ineq+2}
(\bm{\beta}+m_n\bm{\nu})\left(\frac{\bm{\gamma}}{\bm{\beta}}\right)^2
-(\bm{\gamma}+m_n\bm{\varphi})\frac{\bm{\gamma}}{\bm{\beta}}+m_n\bm{\psi}\le 0.
\end{equation}

By calculation, the inequality \eqref{Ineq+2} is equivalent to the
known condition
$$
\bm{\beta} \bm{\gamma} \bm{\varphi}-\bm{\gamma}^2\bm{\nu}-\bm{\beta}^2\bm{\psi} \ge 0.
$$
So we complete the proof of (1).

For (2), $\deg(F(x))=\deg(G(x))=2$. By the assumption that the leading
coefficients of $F(x)$ and $G(x)$ are positive, we have $\bm{\alpha} >0$
and $\bm{\beta}+m_n\bm{\nu} > 0$. Hence for $\bm{\psi}=0$, $F(x) \ll G(x)$ is
reduced to
$$
\bm{\alpha} x^2+\bm{\beta} x+\bm{\gamma} \ll (\bm{\beta}+m_n\bm{\nu})x^2+(\bm{\gamma}+m_n\bm{\varphi})x.
$$
The interlacing is implied by the next inequality
$$
  \bm{\alpha}\left(\frac{\bm{\gamma}+m_n\bm{\varphi}}{\bm{\beta}+m_n\bm{\nu}}\right)^2
- \bm{\beta} \frac{\bm{\gamma}+m_n\bm{\varphi}}{\bm{\beta}+m_n\bm{\nu}}+\bm{\gamma}  \le 0,
$$
that is
$$
m_n(\bm{\beta}+m_n\bm{\nu})(\bm{\beta}\bm{\varphi}-\bm{\gamma}\bm{\nu})-\bm{\alpha}(\bm{\gamma}+m_n\bm{\varphi})^2  \ge 0.
$$
Thus we complete the proof.
\end{proof}

\begin{re}\label{re+SN}
\em
Generally speaking, we mainly consider the polynomial $T_n(x)$ defined
by \eqref{Rec+main} with nonnegative coefficients and the positive
leading coefficients of corresponding $F(x)$ and $G(x)$. Then the stronger
result than Theorem \ref{thm+FG} is that the linear operator $T$ defined by
\eqref{Ope+main} preserves real stability if and only if $F(x)$ and
$G(x)$ have interlacing zeros. In fact, the proof for sufficiency is
similar to Theorem \ref{thm+FG} by the (b) and (c) of Theorem
\ref{thm+RS+Peter} and the proof for necessity can be verified by the
linear operator $T$ acting $(x+w)^{m_n}$ for any $w \in \mathbb{R}$.
\end{re}

In terms of the recurrence relation \eqref{Rec+main}, it is well
known that many combinatorial polynomials can be viewed as the
special case of $T_n(x)$. In what follows, we will apply Theorem
\ref{thm+RS} to the real stability of some combinatorial
polynomials.

Let $a_i, b_i \in \mathbb{R}$ for $i \in [3]$. Define a nonnegative
triangular array $[\mathscr{A}_{n, k}]_{n,k\geq0}$ by
\begin{equation}\label{Rec+A}
 \mathscr{A}_{n,k} =  (a_1n + a_2k + a_3)\mathscr{A}_{n-1,k}
                     +(b_1n + b_2k + b_3)\mathscr{A}_{n-1,k-1}
\end{equation}
for $n\geq1$, where $\mathscr{A}_{0, 0} = 1$ and $\mathscr{A}_{n,k}=0$
unless $0\le k\le n$. For example, $\mathscr{A}_{n,k}$ is the signless
Stirling number of the first kind for $a_1=-a_3=b_3=1$ and the others
are zero and the Stirling number of the second kind for $a_2=b_3=1$
and the others are zero, see \cite{WY06} for more examples. In terms of
the nonnegativity of $[\mathscr{A}_{n,k}]_{n,k\geq0}$, it is natural
to let $a_1n + a_2k + a_3\geq0$ for $n> k\geq0$, which is equivalent to
$$
a_1\geq0,\quad a_1+a_2\geq0,\quad a_1+a_3\geq0.
$$
Let the row-generating function $\mathscr{A}_n(x) = \sum_{k=0}^{n}
\mathscr{A}_{n,k}x^k$. Then we have
\begin{equation}\label{Rec+A+gf}
\mathscr{A}_{n+1}(x) = \left[(b_1n+b_1+b_2+b_3)x+ a_1n+a_1+a_3
\right]\mathscr{A}_n(x)+(b_2x^2+a_2x)D_x\mathscr{A}_n(x),
\end{equation}
where $\deg(\mathscr{A}_n(x)) = n$. Hence, by Theorem \ref{thm+RS}, we
immediately get the following result due to Wang and Yeh
\cite{WY06}.

\begin{coro}\emph{\cite{WY06}}
Let $[\mathscr{A}_{n, k}]_{n,k\geq0}$ be defined by \eqref{Rec+A}.
If $a_1b_2 \le a_2b_1$ and $(a_1+a_3)b_2 \le (b_1+b_2+b_3)a_2$,
then the row-generating function $\mathscr{A}_n(x)$ has only real
zeros for $n \in \mathbb{N}$.
\end{coro}
\begin{proof}
Note that $\mathscr{A}_n(x)$ satisfies the recurrence relation
\eqref{Rec+A+gf}. For the real rootedness of $\mathscr{A}_n(x)$,
taking $\bm{\beta}=b_1n+b_1+b_2+b_3,\bm{\gamma}=a_1n+a_1+a_3,\bm{\varphi}=a_2,\bm{\nu}=b_2$
and $\bm{\psi}=0$ in (1) of Theorem \ref{thm+RS}, it suffices to prove for
$n\geq0$ that
\begin{eqnarray*}
(b_1n+b_1+b_2+b_3)(a_1n+a_1+a_3)a_2-(a_1n+a_1+a_3)^2b_2\geq0,
\end{eqnarray*}
which is obvious from the conditions $a_1b_2 \le a_2b_1$ and
$(a_1+a_3)b_2 \le (b_1+b_2+b_3)a_2$.
\end{proof}

In terms of the recurrence relation \eqref{Rec+A+gf}, we define an
operator $\mathscr{A}$ by
\begin{equation}\label{Oper+A}
   \mathscr{A}
:= \left[(b_1n+b_1+b_2+b_3)x+ a_1n+a_1+a_3\right]I+(b_2x^2+a_2x)D_x.
\end{equation}
By Theorem \ref{thm+FG} and Remark \ref{re+SN}, we know that the
condition in (1) of Theorem \ref{thm+RS} is actually equivalent to
that the operator $\mathscr{A}$ preserves real stability. Thus,
for the operator $\mathscr{A}$, we have the following stronger result.
\begin{prop}\label{prop+oper+A}
Let $U = b_1a_2 - a_1b_2$ and $V = (b_1 + b_2 + b_3)a_2 - (a_1 +
a_3)b_2$. The operator $\mathscr{A}$ defined by \eqref{Oper+A}
preserves real stability if and only if $V + nU \ge 0$.
\end{prop}

\begin{rem}
Proposition \ref{prop+oper+A} implies \cite[Theorem 3.3]{HWY15}.
In fact, in \cite[Theorem 3.3]{HWY15}, Hao et al. assumed that
$$
b_1 \ge 0, \quad b_1 + b_2 \ge 0, \quad b_1 + b_2 + b_3 \ge 0.
$$
The following example indicates that we can drop the restrict
condition $b_1+b_2\geq0$.
\end{rem}

\begin{ex}[\textbf{Andr\'e Polynomials}]
Let $d_{n,k}$ denote the number of the augmented Andr\'e permutations
in $S_n$ with $k-1$ left peaks. Let
$$
D_n(x)=\sum\limits_{k\ge1}d_{n,k}x^k.
$$
It is known that
$$
d_{n+1,k}=kd_{n,k}+(n-2k+3)d_{n,k-1},
$$
where $d_{1,1}=1$, see Foata and Sc\"utzenberger \cite{FS73} and
\cite[A094503]{Slo} for instance. Note that
$$
D_{n+1}(x)=(n+1)xD_n(x)+x(1-2x)D_xD_n(x)
$$
and the degree of $D_n(x)$ is $\lceil n/2 \rceil$. Taking
$\bm{\beta}=n+1, \bm{\gamma}=0, \bm{\nu}=-2, \bm{\varphi}=1$ and
$\bm{\psi}=0$ in (1) of Theorem \ref{thm+RS}, we have that the
operator
$$
D:=(n+1)xI+x(1-2x)D_x
$$
preserves real stability, which implies the real-rootedness of $D_n(x)$.
\end{ex}

As a generalization of the Stirling triangle of the second kind, the
Whitney triangle of the second kind and one triangle of Riordan, the
Stirling-Whitney-Riordan triangle $[\mathscr{S}_{n, k}]_{n,k\geq0}$
satisfies the recurrence relation
\begin{eqnarray}\label{rec+SWR}
      \mathscr{S}_{n,k}
  &=& (b_1k+b_2)\mathscr{S}_{n-1,k-1}+[(2\lambda b_1+a_1)k
      +\lambda(b_1+b_2)+a_2]\mathscr{S}_{n-1,k}            \nonumber\\
  & & +\lambda(a_1+\lambda b_1)(k+1)\mathscr{S}_{n-1,k+1}
\end{eqnarray}
where $\mathscr{S}_{0, 0} = 1$ and $\mathscr{S}_{n, k} = 0 $ unless
$0\le k \le n$, see \cite{Zhu21}. For its row-generating function
$\mathscr{S}_{n}(x)=\sum_{k=0}^{n} \mathscr{S}_{n,k}x^k$, it satisfies
the recurrence relation
\begin{equation}\label{rec+SWR+gf}
   \mathscr{S}_{n}(x)
 = \left[ a_2+(b_1+b_2)(x+\lambda)\right]\mathscr{S}_{n-1}(x)
   +(x+\lambda)\left[a_1+b_1(x+\lambda)\right]D_x\mathscr{S}_{n-1}(x),
\end{equation}
where $\deg(\mathscr{S}_{n}(x)) = n$. By Theorem \ref{thm+RS}, we
get the following result \cite[Theorem 3.2]{Zhu21}.
\begin{coro}\emph{\cite[Theorem 3.2]{Zhu21}}
Let $a_1, a_2, b_1, b_2, \lambda$ be nonnegative. If $a_1(b_1+b_2) \ge a_2 b_1$,
then $\mathscr{S}_{n}(x)$ defined by \eqref{rec+SWR+gf} has only real zeros.
\end{coro}

\begin{proof}
By \eqref{rec+SWR+gf}, we have $F(x)$ and $G(x)$ corresponding to
\eqref{eq+FG} as follows£º
$$
\left\{
\begin{array}{lcc}
 F(x) = (b_1+b_2)x+\lambda(b_1+b_2)+a_2,                                               \\
 \\
 G(x) = (b_1n+b_2)x^2+[(2n-1)\lambda b_1 +\lambda b_2 +(n-1)\lambda a_1+a_2]x
        +(n-1)\lambda(a_1+\lambda b_1).                                                \\
\end{array}
\right.
$$
For the real rootedness of $\mathscr{S}_n(x)$, taking $\bm{\beta}=b_1+b_2,
\bm{\gamma}=\lambda(b_1+b_2)+a_2,\bm{\nu}=b_1, \bm{\varphi}=a_1+2\lambda b_1$ and
$\bm{\psi}=\lambda(a_1+\lambda b_1)$ in (1) of Theorem \ref{thm+RS},
it suffices to prove for $n\geq0$ that
\begin{eqnarray*}
(b_1+b_2)[\lambda(b_1+b_2)+a_2](a_1+2\lambda
b_1)-[\lambda(b_1+b_2)+a_2]^2b_1-(b_1+b_2)^2\lambda(a_1+\lambda
b_1)\geq0.
\end{eqnarray*}
This inequality is equivalent to $a_1(b_1+b_2) \ge a_2 b_1$.
\end{proof}

Based on the classical Eulerian triangle and various triangular
arrays from staircase tableaux, tree-like tableaux and segmented
permutations, Zhu \cite{Zhu20jcta} considered a generalized Eulerian
triangle $[\mathscr{T}_{n,k}]_{n,k\ge 0}$, which satisfies the recurrence
relation:
\begin{eqnarray}\label{rec+gene+Eur}
     \mathscr{T}_{n,k}
 &=& \lambda(a_1k+a_2)\mathscr{T}_{n-1,k}+[(b_1-da_1)n-(b_1-2da_1)k
     +b_2-d(a_1-a_2)]\mathscr{T}_{n-1,k-1}                         \nonumber\\
 & & +\frac{d(b_1-da_1)}{\lambda}(n-k+1)\mathscr{T}_{n-1,k-2},
\end{eqnarray}
where $\mathscr{T}_{0,0}=1$ and $\mathscr{T}_{n,k}=0$ unless
$0 \le k \le n$. In particular, \eqref{rec+gene+Eur} can reduce to
some combinatorial sequences, such as the classical Eulerian numbers
by taking $b_2=d=0$ and $a_1=a_2=b_1=\lambda=1$ and the numbers enumerating
in symmetric tableaux by taking $b_2=d=0, a_1=a_2=\lambda=1$ and $b_1=2$
(see \cite[A109062]{Slo}). We refer the reader to \cite{Zhu20jcta} for more
examples.

We can rewrite \eqref{rec+gene+Eur} by its
row-generating function as follows:
\begin{equation}\label{rec+gene+Eur+gf}
   \mathscr{T}_{n}(x)
 = p_n(x)\mathscr{T}_{n-1}(x)+q_n(x)D_x\mathscr{T}_{n-1}(x),
\end{equation}
where
$$
\left\{
\begin{array}{lll}
  p_n(x) &=&   \frac{(n-1)d(b_1-da_1)}{\lambda}x^2+\left[(n-1)(b_1-da_1)+b_2+da_2 \right]x+\lambda a_2,   \\
         & &   \\
  q_n(x) &=& - \frac{d(b_1-da_1)}{\lambda}x^3-(b_1-2da_1)x^2+\lambda a_1x                                 \\
\end{array}
\right.
$$
and $\deg(\mathscr{T}_n(x)) = n$.

The following result for real rootedness of $\mathscr{T}_n(x)$ proved
in \cite{Zhu20jcta} can easily follow from Theorem \ref{thm+RS}.

\begin{coro}\emph{\cite[Theorem 2.16]{Zhu20jcta}}
Let $a_1, b_1,\lambda$ be positive and $a_2, b_2, d $ be nonnegative.
If $a_2+b_2>0$ and $b_1-da_1 \ge 0$, then the row-generating function
$\mathscr{T}_{n}(x)$ of $[\mathscr{T}_{n,k}]_{n,k}$ in \eqref{rec+gene+Eur}
has only real zeros.
\end{coro}

\begin{proof}

By \eqref{rec+gene+Eur+gf}, we have $F(x)$ and $G(x)$ corresponding
to \eqref{eq+FG} as follows:
$$
\left\{
\begin{array}{lcc}
 F(x) = \frac{(n-1)d}{\lambda}(b_1-da_1)x^2+\left[ (n-1)(b_1-da_1)+b_2+da_2 \right]x+\lambda a_2,      \\
 \\
 G(x) = [(n-1)da_1+b_2+da_2]x^2+\lambda [(n-1)a_1+a_2]x.                                              \\
\end{array}
\right.
$$
Next, we will consider two different cases in terms of $\deg(F(x))$.

Case 1: If $\deg(F(x)) \le 1$, then $d(b_1-da_1)=0$. Furthermore, by
$a_1>0$ and $a_2+b_2>0$, we have $0\le\deg(G(x))-\deg(F(x))\leq1$.
For the real-rootedness of $T_n(x)$, taking $\bm{\beta}=
(n-1)(b_1-da_1)+b_2+da_2, \bm{\gamma}=\lambda a_2, \bm{\nu} =2da_1-b_1,
\bm{\varphi}=\lambda a_1, \bm{\psi}=0$ in (1) of Theorem \ref{thm+RS}, and it
suffices to show
$$
\lambda^2 a_1a_2 [(n-1)(b_1-da_1)+b_2+da_2]-(\lambda
a_2)^2(2da_1-b_1)  \ge 0,
$$
which is equivalent to
$$
(n-1)a_1(b_1-da_1)+a_1b_2+a_2(b_1-da_1) \ge 0.
$$
This is obvious from $a_1>0,b_2\geq0$ and $b_1-da_1\ge0$.

Case 2: If $\deg(F(x))=2$, then $\deg(G(x))=2$. Similarly, taking
$\bm{\alpha} = (n-1)d(b_1-da_1)/\lambda, \bm{\beta}= (n-1)(b_1-da_1)+b_2+da_2,
\bm{\gamma}=\lambda a_2, \bm{\nu} =2da_1-b_1, \bm{\varphi}=\lambda a_1, \bm{\psi}=0$ and
$m_n=n-1$ in (2) of Theorem \ref{thm+RS}. It suffices to show that
\begin{eqnarray*}
& &(n-1)[(n-1)da_1+b_2+da_2][\lambda a_1[(n-1)(b_1-da_1)+b_2+da_2]
   +\lambda a_2(b_1-2da_1)] \\
& &-(n-1)d(b_1-da_1)[\lambda a_2+(n-1)\lambda a_1]^2/\lambda \ge 0.
\end{eqnarray*}
This is equivalent to
$$
(n-1)\lambda  b_2[a_1b_2+a_2b_1+(n-1)a_1b_1] \ge 0.
$$
This inequality follows from nonnegativity of $a_i, b_i$ and $
\lambda$.
\end{proof}


\subsection{Hurwitz stability}
As we know that many combinatorial polynomials have only real zeros.
However, for some other combinatorial polynomials, they don't always
have only real zeros. In this case, they often have all zeros in the
left half-plane, i.e., they are Hurwitz stable. For any univariate
Hurwitz stable polynomial, a nice property is that if its leading
coefficient is positive, then so are all coefficients (see
\cite[Proposition 11.4.2]{RS02}). This is also a useful approach to
verifying the positivity of coefficients of a polynomial.

Let
$$
r(x) = \sqrt{\frac{1+x}{1-x}}.
$$
By induction, one can get
$$
(xD_x)^n(r(x)) = \frac{\mathscr{R}_n(x)}{(1-x)^n(1+x)^{n-1}\sqrt{1-x^2}},
$$
where $\mathscr{R}_n(x) = \sum_{k=0}^{2n-1}\mathscr{R}(n,k)x^k$.
It is easy to know that the polynomial $\mathscr{R}_n(x)$ satisfies
the recurrence relation
\begin{equation}\label{rec+gf+alt+run+Stir+Perm}
     \mathscr{R}_{n+1}(x)
  = (2nx+1)x\mathscr{R}_n(x)+x(1-x^2)D_x\mathscr{R}_n(x)
\end{equation}
for $n \ge 0$, $\mathscr{R}_0(x)=1$ and $\mathscr{R}_1(x) = x$.
For the coefficient $\mathscr{R}(n,k)$, it counts the number of
a dual set of Stirling permutations of order $n$ with $k$
alternating runs, see \cite{MW16}. In addition, in \cite{MW16},
it was found that this polynomial $\mathscr{R}_n(x)$ does not
have only real zeros and proposed the following conjecture.

\begin{conj}\emph{\cite[Conjecture 4.1]{MW16}}\label{conj+MW}
The polynomial $\mathscr{R}_n(x)$ in (\ref{rec+gf+alt+run+Stir+Perm})
is Hurwitz stable for $n \in \mathbb{N}$.
\end{conj}

It is natural to study the Hurwitz stability of combinatorial
polynomials. In the following, we will consider the Hurwitz
stability of $T_n(x)$ in (\ref{Rec+main}) with $\bm{\nu}=\bm{\psi}=0$,
i.e., satisfying the following recurrence relation:
\begin{equation}\label{rec+T+HS}
T_{n+1}(x)=(\bm{\alpha} x^2+\bm{\beta} x+\bm{\gamma})T_{n}(x)+(\bm{\mu} x^3+\bm{\varphi}
x)D_xT_{n}(x),
\end{equation}
where all $\bm{\alpha}, \bm{\beta}, \bm{\gamma}, \bm{\mu}, \bm{\varphi}$
are real sequences in $\mathbb{R}$. In order to show the Hurwitz stability
of $T_n(x)$, we need the following characterization of linear operators
preserving Hurwitz stability of multivariate polynomials, see Borcea and
Br\"and\'en \cite[Remark 7.1]{BB09}.

\begin{thm}\label{thm+HS+peter}
For $n \in \mathbb{N}$, let $\mathbb{T} : \mathbb{C}_n [z] \rightarrow
\mathbb{C}[z]$ be a linear operator. Then $\mathbb{T}$ preserves Hurwitz
stability if and only if either
\begin{itemize}
\item [\rm (a)] $\mathbb{T}$ has range of dimension at most one and is of
                the form $\mathbb{T} (f) = \alpha(f)P$, where $\alpha$ is
                a linear functional on $\mathbb{C}_n [z]$ and $P$ is a
                Hurwitz stable polynomial, or
\item [\rm (b)] the bivariate polynomial
$$
   \mathbb{T}[(1+zw)^n]
:= \sum\limits_{k \le n} \binom{n}{k}\mathbb{T}(z^{k}) w^{k}
$$
\end{itemize}
is Hurwitz stable.
\end{thm}

Our result for Hurwitz stability can be presented as follows.

\begin{thm}\label{thm+HS}
Let $T_n(x)$ be defined by \eqref{rec+T+HS} with all $\bm{\beta}, \bm{\gamma},
\bm{\varphi} \ge 0$ and $T_{n_0}(x)$ be Hurwitz stable. If one of the
followings is true,
\begin{itemize}
\item [\rm (1)] $\deg(T_n(x))=n$ and $\bm{\alpha}=-n\bm{\mu} \ge0$,
\item [\rm (2)] $\deg(T_n(x))=m_n$ $(m_n \neq n)$ and $\bm{\alpha}\ge-m_n\bm{\mu} \ge 0$,
\end{itemize}
then $T_n(x)$ is Hurwitz stable for $n \geq n_0$.
\end{thm}

\begin{proof}

We will present the proof by induction on $n$. By the Hurwitz stable
assumption of $T_{n_0}(x)$, then the statement holds for $n=n_0$.
Let $T = (\bm{\alpha} x^2+\bm{\beta} x+\bm{\gamma})I + (\bm{\mu} x^3+\bm{\varphi} x)D_x$. The
statement for $n\geq n_0+1$ is immediate if the operator $T$
preserves Hurwitz stability. In what follows, we will prove that $T$
preserves Hurwitz stability according to two different cases of
$\deg(T_n(x))$.

(1) If $\deg(T_n(x))=n$, then by \eqref{rec+T+HS}, we have
$\bm{\alpha}+n\bm{\mu}=0$. By Theorem \ref{thm+HS+peter}, it suffices
to show that
\begin{eqnarray*}
  T(1+xy)^n &=& (1+xy)^{n-1}\left[ \bm{\alpha} x^2
                 +(\bm{\beta} x+\bm{\gamma})(1+xy)+n\bm{\varphi} xy \right]                    \\
            &=& (1+xy)^nx\left(\bm{\beta}+\frac{\bm{\gamma}}{x}
                 +\frac{\bm{\alpha} x}{1+xy}+\frac{n\bm{\varphi} y}{1+xy}\right)
\end{eqnarray*}
is Hurwitz stable. Since $(1+xy)^nx$ is Hurwitz stable by
definition, we need to prove that
\begin{equation}\label{fun}
\bm{\beta}+\frac{\bm{\gamma}}{x}+\frac{\bm{\alpha} x}{1+xy}+\frac{n\bm{\varphi} y}{1+xy}
\end{equation}
is Hurwitz stable. Let $\Re(z)$ denote the real part of $z$, where
$z \in \mathbb{C}$. Note that
$$
   \Re\left(\bm{\beta}+\frac{\bm{\gamma}}{x}+\frac{\bm{\alpha} x}{1+xy}
   +\frac{n\bm{\varphi} y}{1+xy}\right)
 = \bm{\beta}+\Re\left(\frac{\bm{\gamma}}{x}\right)
   +\Re\left(\frac{\bm{\alpha}}{\frac{1}{x}+y}\right)
   +\Re\left(\frac{n\bm{\varphi}}{x+\frac{1}{y}}\right).
$$
Whenever $\Re(x) > 0$ and $\Re(y) > 0$, we have $\Re(\frac{1}{x}) > 0$
and $\Re(\frac{1}{y}) > 0$. In consequence, it is obvious that
$$
\Re\left(\frac{\bm{\gamma}}{x}\right)                \ge0,   \quad
\Re\left(\frac{\bm{\alpha}}{\frac{1}{x}+y}\right)    \ge 0,  \quad
\Re\left(\frac{n\bm{\varphi}}{x+\frac{1}{y}}\right)  \ge 0,
$$
since all $\bm{\alpha},\bm{\gamma}, \bm{\varphi} \ge 0$. Hence, by $\bm{\beta} \ge 0$,
the function in \eqref{fun} does not have zeros in the right
half-plane, and thus $T(1+xy)^n$ is Hurwitz stable. In consequence,
$T$ preserves Hurwitz stability.

(2) It is similar to (1). We have
\begin{eqnarray*}
       T(1+xy)^{m_n}
  &=& (1+xy)^{m_n-1}\left[(\bm{\alpha} x^2+\bm{\beta} x+\bm{\gamma})(1+xy)+m_n\bm{\mu} x^3y
      +m_n\bm{\varphi} xy \right]                                                    \\
  &=& (1+xy)^{m_n}x\left[(\bm{\alpha}+m_n\bm{\mu})x+\bm{\beta}+\frac{\bm{\gamma}}{x}
      +\frac{m_n\bm{\varphi} y}{1+xy}-\frac{m_n\bm{\mu} x}{1+xy}\right]              \\
  &=& (1+xy)^{m_n}x\left[(\bm{\alpha}+m_n\bm{\mu})x+\bm{\beta}+\frac{\bm{\gamma}}{x}
      +\frac{m_n\bm{\varphi}}{x+\frac{1}{y}}-\frac{m_n\bm{\mu}}{\frac{1}{x}+y}\right]
\end{eqnarray*}
is Hurwitz stable in terms of the nonnegativity of $\bm{\beta}, \bm{\gamma}, -\bm{\mu},
\bm{\varphi}$ and $\bm{\alpha}+m_n\bm{\mu}$. Hence $T$ preserves Hurwitz stability.
\end{proof}

As an immediate application of Theorem \ref{thm+HS}, we verify
Conjecture \ref{conj+MW} as follows.
\begin{prop}\label{prop+MW}
The polynomial $\mathscr{R}_n(x)$ in (\ref{rec+gf+alt+run+Stir+Perm})
is Hurwitz stable for $n \in \mathbb{N}$.
\end{prop}

\begin{proof}

Obviously, $\mathscr{R}_0(x)=1$ is Hurwitz stable and
$\deg(\mathscr{R}_n(x)) = 2n-1$ due to \eqref{rec+gf+alt+run+Stir+Perm}.
Taking $\bm{\alpha} = 2n, \bm{\beta}=1, \bm{\gamma}=0, \bm{\mu}=-1, \bm{\varphi}=1$
and $m_n=2n-1$ in (2) of Theorem \ref{thm+HS}, we get that the polynomial
$\mathscr{R}_n(x)$ is Hurwitz stable.
\end{proof}

For the classical Eulerian polynomial $A_n(x)$ in
(\ref{Rec+Euler+A}), it is well known that $A_n(x)$ has only real
zeros and $A_{n-1}(x) \ll A_{n}(x)$. Furthermore, it is an
interesting problem to consider the distribution of zeros for some
linear combinations of $A_{n-1}(x)$ and $A_{n}(x)$.  In particular,
Yang and Zhang \cite{YZ15} proved that the following linear
combination:
$$
(x + 1)A_{n-1}(x) + kxA_{n-2}(x)
$$
is Hurwitz stable for $n \ge 2$ and $k \ge -n$. In addition, in
\cite{YZ15}, this Hurwitz stability result played an important role
in proving the interlacing property of the Eulerian polynomials of
between type $D$ and affine type $B$, and a conjecture about the
half Eulerian polynomials of type $B$ and type $D$ proposed by
Hyatt in \cite{Hya16}. As an extension, we will consider the
Hurwitz stability of the next linear combination:
\begin{equation}\label{rec+linear+comb+T}
   (\bm{\varphi}-\bm{\nu} x)\bm{\rho} T_{n+1}(x)
 + [\bm{\varphi} \bm{\eta} x+(\bm{\varphi}-\bm{\gamma})\bm{\varphi} \bm{\rho}]T_{n}(x)
\end{equation}
for $T_{n}(x)$ in (\ref{Rec+main}) with $\bm{\alpha}=\bm{\mu}=\bm{\psi}=0$,
i.e., satisfying the recurrence relation
\begin{equation}\label{rec+gf+two+term+T}
    T_{n+1}(x)
 =  (\bm{\beta} x+\bm{\gamma})T_{n}(x) + (\bm{\nu} x^2+\bm{\varphi} x)D_xT_{n}(x).
\end{equation}
Here $\bm{\rho}$ and $\bm{\eta}$ are abbreviated notation for real sequences
in $\mathbb{R}$. As a consequence of Theorem \ref{thm+HS}, we present the
Hurwitz stability for the linear combination in (\ref{rec+linear+comb+T})
as follows.

\begin{prop}\label{prop+HS+linear+comb}
Let $\deg(T_n(x))=m_n$ and both $\bm{\varphi}$ and $\bm{\rho}$ be nonnegative sequences.
If $(\bm{\beta}+m_n\bm{\nu})\bm{\nu} \le 0$ and $(\bm{\beta} \bm{\varphi}-\bm{\gamma}
\bm{\nu})\bm{\rho}+\bm{\varphi} \bm{\eta} \ge 0$, then the linear combination in
(\ref{rec+linear+comb+T}) is Hurwitz stable for any $n \in
\mathbb{N}$.
\end{prop}

\begin{proof}

By \eqref{rec+gf+two+term+T}, for the linear combination in
(\ref{rec+linear+comb+T}), we have
\begin{eqnarray*}
   & & x\left[(\bm{\varphi}-\bm{\nu} x)\bm{\rho} T_{n+1}(x)
      +[\bm{\varphi} \bm{\eta} x+(\bm{\varphi}-\bm{\gamma})\bm{\varphi} \bm{\rho}]T_{n}(x)\right]\\
 = & &[(\bm{\nu}-\bm{\beta})\bm{\nu} \bm{\rho} x^2+((\bm{\beta} \bm{\varphi}-\bm{\gamma} \bm{\nu})\bm{\rho}
       +\bm{\varphi} \bm{\eta})x](xT_{n}(x))+(-\bm{\nu}^2\bm{\rho}  x^3
       + \bm{\varphi}^2 \bm{\rho} x )D_x(xT_{n}(x)).
\end{eqnarray*}

Then according to the assumption, the Hurwitz stability for the
linear combination in (\ref{rec+linear+comb+T}) follows from Theorem
\ref{thm+HS}.
\end{proof}

\begin{ex}[\textbf{Flower triangle}]
It is known that the flower triangle $[F_{n,k}]_{n,k\ge0}$ satisfies
the following recurrence relation (see \cite[A156920]{Slo}):
$$
F_{n,k} = (1+k)F_{n-1,k} + (2n - 2k + 1)F_{n-1,k-1},
$$
where $F_{0,0}=1$ and $F_{n,k}=0$ unless $0\le k \le n$. Then the
row-generating function $F_n(x)$ satisfies
$$
F_{n+1}(x) = [(2n+1)x+1]F_{n}(x)+x(1-2x)D_xF_{n}(x).
$$
Taking $m_n=n, \bm{\beta}=2n+1, \bm{\gamma}=1, \bm{\nu}=-2$ and
$\bm{\varphi}=1$ in \eqref{rec+gf+two+term+T}. If both $\bm{\rho}$
and $(2n+3)\bm{\rho}+\bm{\eta}$ are nonnegative sequences, then
$$
\bm{\rho}(2x+1)F_{n+1}(x)+\bm{\eta} xF_{n}(x)
$$
is Hurwitz stable for any $n \in \mathbb{N}$.

\end{ex}

\section{The Hurwitz stability of Tur\'an expressions }

In the end of the former section, we consider the Hurwitz stability of
certain linear combination. In this section, we mainly consider the
Hurwitz stability of a non-linear operator.

Given a polynomial sequence $\mathcal{P} = (\mathscr{P}_n(x))_{n \ge 0}$ with
$\deg(\mathscr{P}_n(x)) = n$, we denote the $n$th \emph{Tur\'an expression} by
$$
    \mathfrak{I}_n(\mathcal{\mathscr{P}}; x)
:= (\mathscr{P}_{n+1}(x))^2-\mathscr{P}_{n+2}(x)\mathscr{P}_n(x).
$$
The concept of Tur\'an expression owed to Tur\'an \cite{T50} who
found Tur\'an's inequalities concerning Legendre polynomial sequence
$\mathcal{P}$: $\mathfrak{I}_n(\mathcal{P}; x)\geq0$ for $x \in [-1, 1]$
and $n \in \mathbb{N}$. However, it was first published by Szeg\"o
\cite{Sz48}. We refer the reader to \cite{CGV15, Z18ejc} and references
therein for more information about Tur\'an's inequalities.
We say that $(\mathscr{P}_n(q))_{n \ge 0}$ is {\it $q$-log-concave}
(resp., {\it $q$-log-convex}) if all coefficients of
$\mathfrak{I}_n(\mathcal{P}; q)$ (resp., $-\mathfrak{I}_n(\mathcal{P}; q)$)
are nonnegative. The definition of the $q$-log-concavity was first
suggested by Stanley and that of the $q$-log-convexity was first
introduced Liu and Wang. Note the fact that if a univariate polynomial
is Hurwitz stable, then the signs of its all coefficients are same.
Thus the Hurwitz stability of a Tur\'an expression implies that the
original polynomial sequence is either $q$-log-concave or $q$-log-convex.

It is known that both the classical Eulerian polynomials and Bell
polynomials are $q$-log-convex \cite{LW07}. Moreover, their Tur\'an
expressions are Hurwitz stable. For many other combinatorial polynomials,
including the Eulerian polynomials of types $B$, Lah polynomials,
descent polynomials on segmented permutations, and so on, their
Tur\'an expressions are also Hurwitz stable, see \cite{CGV15,Fi06,
Z18ejc,Zhu20jcta,Zhu21}. In this section, we will derive a new
criterion for the Hurwitz stability of Tur\'an expression. Then
we apply this criterion to many combinatorial polynomials in a
unified manner. The following result for two interlacing polynomials
plays an important role in our proof.
\begin{lem}\emph{\cite[Lemma 1.20]{Fi06}}\label{lem+Fisk}
Let both $f(x)$ and $g(x)$ be standard real polynomials with only
real zeros. Assume that $\deg(f(x))=n$ and all real zeros of $f(x)$
are $r_1, \ldots, r_n$. If $\deg(g)=n-1$ and we write
$$g(x)=\sum_{i=1}^n\frac{c_if(x)}{x-r_i},$$
then $g\ll f$ if and only if all $c_i$ are nonnegative.
\end{lem}

Let $(\mathscr{P}_n(x))_{n \ge 0}$ be a sequence of polynomials with
nonnegative coefficients and satisfy the recurrence relation
\begin{equation}\label{Rec+Turan}
  \mathscr{P}_{n+1}(x)=p_n(x)\mathscr{P}_n(x)+q(x)D_x\mathscr{P}_n(x),
\end{equation}
where $\deg(\mathscr{P}_{n}(x))=\deg(\mathscr{P}_{n-1}(x))+1$.
Denote by $\{r_k\}_{k=1}^n$ all zeros of $\mathscr{P}_n(x)$ and
define
\begin{equation}\label{dec+root}
(x-r_k)\left[p_n(x)-p_{n-1}(x)\right]+q(x):=h_n(x)\sum_{i=0}^3
a_{k_i}(x-r_k)^i
\end{equation}
for $1\leq k\leq n$, where $h_n(x)$ is a polynomial.

The main result of this section can be stated as follows.
\begin{thm}\label{thm+Turan}
Let $\mathscr{P}_n(x)$ be defined by \eqref{Rec+Turan} and
$\mathscr{P}_n(x) \ll \mathscr{P}_{n+1}(x)$. Assume that
$h_n(x)$ is Hurwitz stable for each $n$. If all elements of
$\bigcup_{k=1}^n\{-a_{k_3}, a_{k_1}, a_{k_0} \}$ have same sign,
and the right side of \eqref{dec+root} has same sign for
$1\leq k\leq n$ and $x > 0$, then $\mathfrak{I}_n(\mathcal{P}; x)$
is Hurwitz stable for
each $n$.
\end{thm}

\begin{proof}
In terms of the hypothesis $\mathscr{P}_n(x) \ll \mathscr{P}_{n+1}(x)$,
all the zeros $r_k$ of $\mathscr{P}_n(x)$ are real and non-positive and by Lemma
\ref{lem+Fisk} we can write
\begin{eqnarray}\label{quotient+T1}
 \frac{\mathscr{P}_{n-1}(x)}{\mathscr{P}_n(x)}=\sum_{i=1}^n\frac{t_i}{x-r_i},
\end{eqnarray}
where all $t_i$ are nonnegative. Furthermore, we have
\begin{eqnarray}\label{quotient+T2}
   D_x\left(\frac{\mathscr{P}_{n-1}(x)}{\mathscr{P}_n(x)}\right)
 = \sum_{i=1}^n\frac{-t_i}{(x-r_i)^2}.
\end{eqnarray}

By \eqref{Rec+Turan}-\eqref{quotient+T2}, we get
\begin{eqnarray*}
  \mathfrak{I}_n(\mathcal{P}; x)
  &=& \left[p_n(x)\mathscr{P}_n(x)+q(x)D_x\mathscr{P}_n(x)\right]\mathscr{P}_{n-1}(x)
      - \mathscr{P}_n(x)\left[p_{n-1}(x)\mathscr{P}_{n-1}(x)+q(x)D_x\mathscr{P}_{n-1}(x)\right]              \\
  &=& \left[p_n(x)-p_{n-1}(x)\right]\mathscr{P}_n(x)\mathscr{P}_{n-1}(x)
      +q(x)\left[\mathscr{P}_{n-1}(x)D_x\mathscr{P}_n(x)-\mathscr{P}_{n}(x)D_x\mathscr{P}_{n-1}(x)\right]    \\
  &=& \mathscr{P}_n^2(x)\left[\left[p_n(x)-p_{n-1}(x)\right]\frac{\mathscr{P}_{n-1}(x)}
      {\mathscr{P}_n(x)}-q(x)D_x\left(\frac{\mathscr{P}_{n-1}(x)}{\mathscr{P}_n(x)}\right)\right]            \\
  &=& \mathscr{P}_n^2(x)\sum\limits_{k=0}^n \frac{t_k\left[(x-r_k)(p_n(x)
      -p_{n-1}(x))+q(x)\right]}{(x-r_k)^2}                                                                   \\
  &=& \mathscr{P}_n^2(x) h_n(x)\sum\limits_{k=0}^n t_k\left[a_{k_3}(x-r_k)
      + a_{k_2}+\frac{a_{k_1}}{x-r_k}+\frac{a_{k_0}}{(x-r_k)^2}\right].
\end{eqnarray*}

Obviously, $\mathscr{P}_n(x)$ and $h_n(x)$ are Hurwitz stable. Thus we will
consider the following function:
\begin{equation}\label{Des+tl}
  a_{k_3}(x-r_k)+a_{k_2}+\frac{a_{k_1}}{x-r_k}+\frac{a_{k_0}}{(x-r_k)^2}.
\end{equation}

Without loss of generality, we assume that
$\bigcup_{k=1}^n\{-a_{k_3}, a_{k_1},a_{k_0}\}$ has positive (resp.,
neagtive) sign. If $\Re(x)>0$ and $\Im(x)\neq0$, then, obviously,
for the image part of \eqref{Des+tl}, we derive
\begin{eqnarray*}
 \Im(x) \Im\left(a_{k_3}(x-r_k)+a_{k_2}+\frac{a_{k_1}}{x-r_k}+
 \frac{a_{k_0}}{(x-r_k)^2}\right)<0 \quad (\text{resp.,} > 0)
\end{eqnarray*}
for all $k\in[n]$. Hence $\Im\left(\sum_{k=1}^n
t_k\left[a_{k_3}(x-r_k) + a_{k_2} +\frac{a_{k_1}}{x-r_k} +
\frac{a_{k_0}}{(x-r_k)^2}\right]\right)\neq0$ for $\Re(x)>0$ and
$\Im(x)\neq0$.

If $\Re(x)>0$ and $\Im(x)=0$, then, by hypothesis, we have that
all signs of
\begin{eqnarray*}
  a_{k_3}(x-r_k) + a_{k_2} +\frac{a_{k_1}}{x-r_k} +
 \frac{a_{k_0}}{(x-r_k)^2}
\end{eqnarray*}
are same for all $k\in[n]$. In consequence, $\sum_{k=1}^n
t_k\left[a_{k_3}(x-r_k) + a_{k_2} +\frac{a_{k_1}}{x-r_k} +
\frac{a_{k_0}}{(x-r_k)^2}\right]\neq0$.

Hence, from above two cases, we get that
$\mathfrak{I}_n(\mathcal{P}; x)$ is nonzero when $\Re(x)
> 0$. Namely $\mathfrak{I}_n(\mathcal{P}; x)$ is Hurwitz stable
for each $n$.

\end{proof}

\begin{re}\label{rem+Turan}
\em Obviously, the conclusion for $\mathfrak{I}_n(\mathcal{P}; x)$
in Theorem \ref{thm+Turan} can be extended to that
$\mathfrak{I}_n(\mathcal{P}; x+z_n)$ is Hurwitz stable if $z_n$ is
not less than the largest zero of $\mathscr{P}_n(x)$ for all nonnegative
integers $n$.
\end{re}

\subsection{The generalized Eulerian polynomials}

Fisk showed that the Tur\'an expressions of Eulerian polynomials are
Hurwitz stable in his unfinished book (see \cite[Lemma
21.91]{Fi06}), but his proof is incorrect. In \cite{Z18ejc}, Zhu
again proved the Hurwitz stablity of Eulerian polynomials. And here,
we will give a generalized result.

For $r \ge 1$, Riordan \cite{Rio58} defined the $r$-Eulerian
polynomial
\begin{eqnarray*}\label{def+r+Euler}
E_{n,r}(x) = \sum\limits_{\pi \in S_n} x^{\text{exc}_r(\pi)},
\end{eqnarray*}
where $\text{exc}_r(\pi)$, the number of $r$-excedances of $\pi$,
is defined by
$$
\text{exc}_r(\pi) =| \{i \in [n] : \pi_i \ge i + r\}|.
$$
And then, Riordan \cite[p. 214]{Rio58} got the following recurrence
relation:
\begin{equation}\label{Rec+exc+poly}
E_{n,r}(x)=[(n-r)x+r]E_{n-1,r}(x)+x(1-x)D_xE_{n-1,r}(x),
\end{equation}
where $E_{r,r}(x)=r!$ for $n\ge r$. Note that whenever $r=1$, $E_{n,1}(x)$
is the classical Eulerian polynomial.

In addition, to study the volume of the usual permutohedron, Postnikov
\cite{Pos09} introduced the \emph{mixed Eulerian numbers} $A_{a_1,\dots,a_r}$,
where $a_i\ge0$ and $a_1+\cdots+a_r=r$. In terms of mixed Eulerian numbers,
Berget et al. \cite{BST20} defined the polynomial
$$
    A_{a_1,\dots,a_r}(x)
 := \sum\limits_{i=0}^{n-r}A_{0^i,a_1,\dots,a_r,0^{n-r-i}}x^i
$$
for $a_i\ge1$ and $a_1+\cdots+a_r=n$. And then, they gave the recurrence
relation of the polynomial $A_{a_1,\dots,a_r}(x)$ as follows:
\begin{equation}\label{Rec+mix+Euler}
    A_{a_1,\dots,a_{r}+1}(x)
  = [(n-r+1)x+r]A_{a_1,\dots,a_r}(x)+x(1-x)D_xA_{a_1,\dots,a_r}(x).
\end{equation}
Note that $A_{a_1,\dots,a_r}(x)$ is the $r$-Eulerian polynomial
$E_{n,r}(x)$ whenever $a_i=1$ for $i\in[r-1]$ and $a_r=n-r+1$.
In particular, it follows from the recurrence relation of
$A_{a_1,\dots,a_r}(x)$ that $A_{a_1,\dots,a_r}(1)=n!$ which was
conjectured by Stanley and proved by Postnikov (see \cite[Theorem
16.4]{Pos09}). On the other hand, by using the method of zeros
interlacing, it is easy to know that $A_{a_1,\dots,a_{r}}(x)$ has
only non-positive zeros, moreover, zeros of $A_{a_1,\dots,a_{r}}(x)$
interlace those of $A_{a_1,\dots,a_{r}+1}(x)$. Thus, the
coefficients of $A_{a_1,\dots,a_{r}}(x)$ are unimodal and
log-concave.

In addition, we define the following polynomial
\begin{equation}\label{Rel+mix+Euler}
\mathcal{J}_{n,r}^{a_1,\dots,a_{r}}(x):=\frac{x^{n-r}A_{a_1,\dots,a_{r}}(1/x)}{r!}.
\end{equation}
Combining \eqref{Rec+mix+Euler} and \eqref{Rel+mix+Euler}, it is
easy to know that $\mathcal{J}_{n,r}^{a_1,\dots,a_{r}}(x)$ satisfies
the following relation
\begin{equation}\label{Rel+rec+mix+Euler}
   \mathcal{J}_{n,r}^{a_1,\dots,a_{r}}(x)
 = [(n-1)x+1]\mathcal{J}_{n-1,r}^{a_1,\dots,a_{r}}(x)
   +x(1-x)D_x\mathcal{J}_{n-1,r}^{a_1,\dots,a_{r}}(x),
\end{equation}
where $\mathcal{J}_{n,r}^{a_1,\dots,a_{r}}(1)=n!/r!$.

Let $\mathcal{J}_{n,r}$ be the set of injections $\pi:[n-r]
\rightarrow [n]$. Based on the set of images of $\pi$, define
the polynomial
\begin{equation*}
\mathcal{J}_{n,r}(x)=\sum\limits_{\pi\in\mathcal{J}_{n,r}}x^{\text{exc}(\pi)},
\end{equation*}
where $\text{exc}(\pi) = \text{exc}_1(\pi)$. Then, a relation
between $E_{n,r}(x)$ and $\mathcal{J}_{n,r}(x)$ was proved in
\cite{Rio58} as follows:
\begin{equation}\label{Rea+J+A}
\mathcal{J}_{n,r}(x)=\frac{x^{n-r}E_{n,r}(1/x)}{r!}.
\end{equation}
Combining \eqref{Rec+exc+poly} and \eqref{Rea+J+A}, Elizalde
\cite{Eli21} gave
\begin{equation*}
   \mathcal{J}_{n,r}(x)
 = [(n-1)x+1]\mathcal{J}_{n-1,r}(x)+x(1-x)D_x\mathcal{J}_{n-1,r}(x)
\end{equation*}
for $n > r$. where $\mathcal{J}_{r,r}(x)=1$. Note that $\mathcal{J}_{n,r}(x)
=\mathcal{J}_{n,r}^{1,\dots,1,n-r+1}(x)$. This recurrence relation is the
same as that of the classical Eulerian polynomials, but the initial condition
is different. For $r \in \{2, 3, 4, 5\}$, the reader can be referred to
\cite[A144696-A144699]{Slo}. Obviously, $\mathcal{J}_{n+r,r}(x)$ is a
special case of the generalized Eulerian polynomial $\mathscr{T}_n(x)$ in
(\ref{rec+gene+Eur+gf}) by taking $d=0$ and $\lambda=1$.

Archer et.al \cite{AGPS19} introduced the quasi-Stirling
permutations $\bar{Q}_n$, which is a set of
$\pi=\pi_1\pi_2\cdots\pi_n$ in the multiset $\{1, 1, 2, 2, \dots, n,
n\}$ avoiding $1212$ and $2121$, i.e., there does not exist
$i<j<k<\ell$ such that $\pi_i=\pi_k$ and $\pi_j=\pi_\ell$ for any
$\pi \in \bar{Q}_n$. Elizalde \cite{Eli21} defined the quasi-Stirling
polynomial
\begin{equation*}
\bar{Q}_n(x)=\sum\limits_{\pi \in \bar{Q}_n} x^{\text{des}(\pi)}
\end{equation*}
and he got $\bar{Q}_n(x)=\mathcal{J}_{2n,n+1}(x)$.

These different kinds of Eulerian polynomials can be obtained by
the transformation of the special cases of the generalized Eulerian
polynomial $\mathscr{T}_n(x)$ in (\ref{rec+gene+Eur+gf}) by taking
$d=0, \lambda=1$ and different initial conditions. Applying Theorem
\ref{thm+Turan} to the generalized Eulerian polynomial $\mathscr{T}_n(x)$,
we get the following result proved by Zhu \cite{Zhu20jcta}.

\begin{coro}\emph{\cite[Theorem 2.16]{Zhu20jcta}}\label{Coro+Z+Turan+gen+Eur}
Let $\mathscr{T}=(\mathscr{T}_n(x))_{n\geq0}$, where $\mathscr{T}_n(x)$
is the $n$-row generating function of the generalized Eulerian triangle
in (\ref{rec+gene+Eur}).If $\{ a_1, b_1, \lambda\} \subseteq
\mathbb{R}^{ > 0}$ and $\{a_2, b_2, d\} \subseteq \mathbb{R}^{\ge 0}$
with $a_2 + b_2 > 0$, then $\mathfrak{I}_n(\mathscr{T}; x)$ is Hurwitz
stable for all $n$.
\end{coro}

\begin{proof}

It was proved for $n \in \mathbb{N}$ that $\mathscr{T}_n(x) \ll
\mathscr{T}_{n+1}(x)$ and all zeros of $\mathscr{T}_n(x)$ are in $[-\lambda/d, 0]$
in \cite{Zhu20jcta}. By the recurrence relation
(\ref{rec+gene+Eur+gf}), we derive the corresponding
(\ref{dec+root}) as follows:
$$
  (x-r_k)[p_n(x)-p_{n-1}(x)]+q(x)
= \frac{x}{\lambda}(dx+\lambda)[(\lambda + dr_k)a_1-b_1r_k].
$$
We take $h_n(x)=x(dx+\lambda)/\lambda$ and
$a_{k_0}=(\lambda+dr_k)a_1-b_1r_k$. By the assumption conditions and
$r_k \in [-\lambda/d, 0]$, then the desired result is immediate by
Theorem \ref{thm+Turan}.
\end{proof}

\begin{re}
\em By \eqref{Rea+J+A}, the Tur\'an expressions of polynomial
sequence $(E_{n+r,r}(x))_{n\ge0}$ are also Hurwitz stable. So are
those of $(A_{a_1,\dots,a_{r}+n}(x))_{n\ge0}$ since
$\mathcal{J}_{n,r}^{a_1,\dots,a_{r}}(x)$ is the special case of the
generalized Eulerian polynomial $\mathscr{T}_n(x)$ in
(\ref{rec+gene+Eur+gf}) by taking $d=0$ and $\lambda=1$.

Note that the exponential generating function of the mixed Eulerian
numbers $A_{a_1,\dots,a_r}$ is the volume ${\rm Vol} P_{r+1}$ of a
permutohedron $P_{r+1}$ (see \cite[Section 16]{Pos09} for details).
In fact, ${\rm Vol} P_{r+1}$ is a Lorentzian polynomial by using the
conclusion in \cite{BH20}. Thus, ${\rm Vol} P_{r+1}$ has the
corresponding properties, such as the M-convexity of $supp({\rm Vol}
P_{r+1})$ and discrete log-concavity of the mixed Eulerian numbers
$A_{a_1,\dots,a_r}$.
\end{re}

\subsection{The generalized Bell polynomials}
For the Bell polynomial $B_n(x)$, it satisfies the recurrence
relation
\begin{eqnarray*}
B_{n+1}(x)=xB_n(x)+xD_xB_n(x), &\text{where} \quad B_0(x)=1.
\end{eqnarray*}
Fisk showed that Bell polynomials are Hurwitz stable in his
unfinished book (see \cite[Lemma 21.92]{Fi06}). But Chasse et al.
pointed out that Fisk's proof is incorrect and reproved that of Bell
polynomials in \cite{CGV15}. In fact, Chasse et al. proved the Hurwitz
stability of Tur\'an expression for the generalized Bell polynomials
in the following result, which follows from Theorem \ref{thm+Turan}.

\begin{coro}\emph{\cite[Theorem 1.1]{CGV15}}\label{coro+Chass+turan}
Let $\mathcal{B}=(\mathscr{B}_n(x))_{n \ge 0}$ be a real polynomial
sequence with $\deg(\mathscr{B}_n(x))=n$. If $\mathscr{B}_n(x)$
satisfies
\begin{eqnarray}\label{rec+dif+T}
\mathscr{B}_{n+1}(x)=a(x+b)(c_n+D_x)\mathscr{B}_n(x),
\end{eqnarray}
where $a \neq 0, b \ge 0$ and $c_{n+1} \ge c_n > 0$ for all $n \in
\mathbb{N}$, then $\mathfrak{I}_n(\mathcal{\mathcal{B}}; x-b)$ is
Hurwitz stable for all $n\in \mathbb{N}$.
\end{coro}

\begin{proof}

Without loss of generality, we assume $a > 0$ and $\mathscr{B}_0(x) > 0$.
Obviously \eqref{rec+dif+T} implies that all coefficients of
$\mathscr{B}_n(x)$ are real and nonnegative for $n\in \mathbb{N}$.
By induction on $n$, we can show that $\mathscr{B}_n(x)$ is
real-rooted with all zeros $r_k\leq-b$ for $k \in [n]$ and
$\mathscr{B}_n(x) \ll \mathscr{B}_{n+1}(x)$ for all $n \in \mathbb{N}$.

Then the corresponding \eqref{dec+root} for $\mathscr{B}_n(x)$ is
$$
(x-r_k)(p_n(x)-p_{n-1}(x))+q(x)=a(x+b)\left[(c_n-c_{n-1})(x-r_k)+1\right].
$$
We can take $h_n(x) =a(x+b)$, $ a_{k_0}=1$ and $a_{k_1}=c_n-c_{n-1}$.
Hence $\mathfrak{I}_n(\mathcal{\mathcal{B}}; x)$ is Hurwitz stable by
Theorem \ref{thm+Turan}. So is $\mathfrak{I}_n(\mathcal{\mathcal{B}}; x-b)$
by Remark \ref{rem+Turan}.
\end{proof}

\subsection{The Stirling-Whitney-Riordan polynomials}

For the Tur\'an expression of the row-generating function of the
Stirling-Whitney-Riordan triangle (\ref{rec+SWR}), Zhu \cite{Zhu21}
proved the following result concerning its Hurwitz stability. It can
also be looked as a corollary of Theorem \ref{thm+Turan}.

\begin{coro}\emph{\cite[Theorem 3.2]{Zhu21}}\label{cor+Z+Turan+SWR}
Let $\mathcal{S}=(\mathscr{S}_n(x))_{n\geq0}$, where $\mathscr{S}_n(x)$
is the $n$-th row-generating function of the Stirling-Whitney-Riordan
triangle in \eqref{rec+SWR}.
 If $\{\lambda, a_1, a_2, b_1,
b_2\}\subseteq\mathbb{R}^{\ge 0}$ and $a_1(b_1 + b_2) \ge a_2b_1$, then
$\mathfrak{I}_n(\mathcal{S}; x-\lambda)$ is Hurwitz stable for all
$n$.
\end{coro}

\begin{proof}
Note that it was proved in \cite{Zhu21} that all zeros of $\mathscr{S}_n(x)$
are in $(-\lambda-a_1/b_1, -\lambda)$ and $\mathscr{S}_{n-1}(x) \ll
\mathscr{S}_{n}(x)$ for all $n \in \mathbb{N}$. Thus, by Remark
\ref{rem+Turan}, it suffices to show that $\mathfrak{I}_n(\mathcal{S}; x)$
is Hurwitz stable for all $n$.

By (\ref{rec+SWR+gf}), we get the corresponding (\ref{dec+root})
for $\mathscr{S}_n(x)$ as follows:
$$
   (x-r_k)(p_n(x)-p_{n-1}(x))+q(x)
 = (x+\lambda)\left[a_1+b_1(x+\lambda)\right].
$$
By taking $h_n(x) =(x+\lambda)\left[a_1+b_1(x+\lambda) \right]$
and $a_{k_0}=1$, the desired result concerning Hurwitz stability
is immediate by Theorem \ref{thm+Turan}.
\end{proof}

There exists some combinatorial polynomials such that Corollaries
\ref{Coro+Z+Turan+gen+Eur}, \ref{coro+Chass+turan} and \ref{cor+Z+Turan+SWR}
can not be used. But our Theorem \ref{thm+Turan} is still valid.
Some such examples are given in the following.

\subsection{Alternating runs of type  \texorpdfstring{$A$} a}

We say that $\pi \in S_n$ changes direction at position $i$ if
either $\pi_{i-1} < \pi_i > \pi_{i+1}$ or $\pi_{i-1}
> \pi_i < \pi_{i+1}$ for $i \in \{2, \dots, n-1\}$.
Let $R(n, k)$ be the number of $\pi \in S_n$ having $k$ alternating
runs, namely there are $k-1$ indices $i$ such that $\pi$ changes
direction at these positions. For example, let $\pi = 31264875$ and
its alternating runs are $312, 264, 648$. Andr\'e \cite{An84} gave
the recurrence relation as follows:
\begin{equation}\label{Rec+ar+A}
R(n, k) = kR(n- 1, k) + 2R(n-1, k-1) + (n-k)R(n-1, k-2)
\end{equation}
for $n, k \ge 1$, where $R(1,0)=1$ and $R(1,k)=0$ for $k \ge 1$.
Let the row-generating function $R_n(x)=\sum_{k=0}^{n}R(n,k)x^k$.
Then the recurrence relation \eqref{Rec+ar+A} implies
$$
R_{n+2}(x) = x(nx + 2)R_{n+1}(x) + x(1-x^2)D_xR_{n+1}(x)
$$
with $R_1(x)=1$ and $R_2(x)=2x$. Zhu \cite{Z18ejc} proved the
$q$-log-convexity of $R_n(q)$, which is immediate from the
following stronger result.

\begin{prop}\label{prop+alt+run+turan+A}
The Tur\'an expressions of $(R_n(x))_{n \ge 0}$ are Hurwitz stable.
\end{prop}

\begin{proof}

We know that all zeros of $R_n(x)$ are in $[-1, 0]$ and $R_n(x)
\ll R_{n+1}(x)$ (see Ma and Wang \cite{MW08} for the details).
Then the corresponding \eqref{dec+root} for $R_n(x)$ is
$$
(x-r_k)(p_n(x)-p_{n-1}(x))+q(x) = x[-r_k(x-r_k)+1-r_k^2].
$$
Taking $h_n(x)=x, a_{k_0}=1-r_k^2$ and $a_{k_1}=-r_k$. The desired
result follows from Theorem \ref{thm+Turan} since $r_k \in [-1, 0]$.
\end{proof}

\subsection{The longest alternating subsequences and up-down runs}

Let $\widetilde{\pi} = \pi_{i_1} \cdots \pi_{i_k}$ be a subsequence
of $\pi \in S_n$. We say $\widetilde{\pi}$ is an \emph{alternating
subsequence} of $\pi$ if $\widetilde{\pi}$ satisfies
$$
\pi_{i_1} >  \pi_{i_2} < \pi_{i_3} > \cdots \pi_{i_k}.
$$
Denote by $a(n,k)$ the number of $\pi \in S_n$, where the length of
the longest alternating subsequence of $\pi$ is $k$. B\'ona
\cite[Section 1.3.2]{B12} showed that the row-generating function
$t_n(x) =\sum_{k=0}^{n} a(n,k)x^k$ satisfies the following identity:
$$
t_n(x)= \frac{1}{2}(1+x)R_n(x)
$$
for $n \ge 2$. In addition, $t_0(x)=1$ and $t_1(x) = x$.

Note that $t_n(x)$ coincides with the up-down runs polynomial, see
\cite[A186370]{Slo}. In addition, $t_n(x)$ is closely related to
two kinds of peak polynomials $W_n(x)$ and $\widetilde{W}_n(x)$,
which are defined by
\begin{eqnarray}\label{left+peak}
     W_{n}(x)
 &=& \sum_{\pi\in S_n}x^{pk(\pi)}
  =  \sum_{k\geq0}W_{n,k}x^{k},     \nonumber \\
     \widetilde{W}_{n}(x)
 &=& \sum_{\pi\in S_n}x^{lpk(\pi)}
  =  \sum_{k\geq0}\widetilde{W}_{n,k}x^{k},
\end{eqnarray}
where $W_{1}(x)=1, \widetilde{W}_{0}(x)=1$ and $pk(\pi)$ and $lpk(\pi)$
denote the number of interior peaks and left peaks of
$\pi\in S_n$, respectively, see Petersen \cite{Pet07}, Stembridge \cite{St97}
and \cite[A008303, A008971]{Slo} for instance .

Based on these, Ma \cite{M12} defined the polynomials $M_n(x)$ by
\begin{eqnarray}\label{Rec+peak+HS}
M_n(x) = xW_n(x^2)+\widetilde{W}_n(x^2),
\end{eqnarray}
where $M_1(x)=1+x$. In fact, the coefficients of $M_n(x)$ arise
in expansion of $n$-th derivative of $\tan(x)+\sec(x)$, see
\cite[A198895]{Slo}.
It is known that $M_n(x)$ satisfies the
recurrence relation
$$
M_{n+1}(x)=(nx^2+1)M_n(x)+x(1-x^2)D_xM_n(x).
$$
It was shown that all zeros of $M_n(x)$ are in $[-1, 0]$ and $M_n(x)
\ll M_{n+1}(x)$ in \cite{M12}. Thus, by Proposition
\ref{prop+alt+run+turan+A} or Theorem \ref{thm+Turan}, we
immediately have the following result, which in particular implies
$q$-log-convexity of $(t_n(q))_{n \ge 0}$ and $(M_n(q))_{n\ge0}$
\cite{Z18ejc}.

\begin{prop}
The Tur\'an expressions of $(t_n(x))_{n \ge 0}$ and $(M_n(x))_{n \ge
0}$ are both Hurwitz stable.
\end{prop}




\subsection{Alternating runs of type \texorpdfstring{$B$} b}

Now, we consider the alternating runs of type $B$.
Let $B_n$ be all signed permutations of the set $\pm[n]$ such
that $\pi(-i) = -\pi(i)$ for all $i \in [n]$, where $\pm[n] =
\{\pm1, \pm2, \dots , \pm n\}$.
We say that $\pi
\in B_n$ is a \emph{alternating run} if $\pi_{i-1}<\pi_i>\pi_{i+1}$
or $\pi_{i-1}>\pi_i<\pi_{i+1}$ for $i\in [n-1]$ in the order
$\cdots<\overline{2}<\overline{1}<0<1<2<\cdots$, where $\pi_0 = 0$.
Taking the subset $B^{u}_n \subseteq B_n$, which satisfies $\pi_1 > 0$
whenever $\pi \in B^{u}_n$. We call $B^{u}_n$ the {\it up signed
permutations}. For example, taking $\pi =
31\overline{2}\overline{6}48\overline{7}5$, whose alternating runs
is $\{31, \overline{2}\overline{6}4, 48\overline{7}, 8\overline{7}5
\}$. Let $Z(n, k)$ denote the number of up signed permutations $\pi
\in B^{u}_n$ having $k-1$ alternating runs. Zhao \cite{Zha11} got the
following recurrence relation:
\begin{equation}\label{Rec+ar+B}
Z(n, k) = (2k-1)Z(n-1, k) + 3Z(n-1, k-1) + (2n-2k + 2)Z(n-1, k-2)
\end{equation}
for $n \ge 2$ and $k \in [n]$, where $Z(1, 1) = 1$ and $Z(1, k) = 0$
for $k > 1$.

Let the row-generating function $Z_n(x) = \sum_{k=1}^{n}Z(n,k)x^k$.
Then the recurrence relation \eqref{Rec+ar+B} implies
$$
Z_n(x) = [(2n-2)x^2+3x-1]Z_{n-1}(x)+2x(1-x^2)D_xZ_{n-1}(x),
$$
where $Z_1(x)=x$ and $Z_2(x)=x+3x^2$. It was proved in \cite{Z18ejc}
that $(Z_n(q))_{n\geq1}$ is $q$-log-convex, which is also immediate
from the following stronger result.

\begin{prop}
The Tur\'an expressions of $(Z_n(x))_{n \ge 1}$ are Hurwitz stable.
\end{prop}

The proof is similar to that of Proposition
\ref{prop+alt+run+turan+A}, thus we omit it for brevity.

\section{Semi-\texorpdfstring{$\gamma$-} ppositivity and Hurwitz stability}

The location of zeros of polynomials implies much information. For
example, the well-known Newton inequalities say that if all zeros of
a polynomial are real and nonpositive, then its coefficients are
log-concave and unimodal. Moreover, Br\"and\'en in \cite{Bra04}
proved that if all zeros of a symmetric polynomial are real and
nonpositive, then the polynomial has $\gamma$-positivity. In this
section, we will demonstrate a similar result concerning Hurwitz
stability and semi-$\gamma$-positivity.

For $f(x)=\sum_{k=0}^n f_k x^k \in \mathbb{R}[x]$, we say $f(x)$ is
{\it unimodal} if there exists $ m$ such that $f_0\le
f_1\le\cdots\le f_m\ge\cdots\ge f_{n-1}\ge f_n$ and is {\it
symmetric} if $f_k=f_{n-k}$ for $0\leq k\leq n$. Clearly, $f(x)$ is
symmetric if and only if $f(x)=x^nf(1/x)$. We know that any
symmetric polynomial $f(x)$ has the following decomposition:
$$
f(x)= \sum\limits_{k=0}^{\left \lfloor n/2 \right \rfloor} g_k
x^k(1+x)^{n-2k}.
$$
If $g_k \ge 0$ for all $0 \le k \le n$, then we say that $f(x)$ is
{\it $\gamma$-positive}. In particular, $\gamma$-positivity implies
unimodality. Furthermore, in terms of parity of $n$, one can write
$f(x)$ as
\begin{eqnarray*}
  f(x) &=& (1+x)^{\chi(n~mod~2)}\sum\limits_{k=0}^{\left \lfloor n/2 \right \rfloor} g_k x^k(1+x^2)^{\left
           \lfloor n/2 \right \rfloor-k}.
\end{eqnarray*}

Based on these, Ma et al. \cite{MMY20} introduced the following
semi-$\gamma$-positivity.
\begin{definition}
Let $\nu=0$ or $1$. If a polynomial
\begin{eqnarray}\label{sem+des}
 f(x) &=& (1+x)^{\bm{\nu}}\sum_{k=0}^n g_k x^k(1+x^2)^{n-k}
\end{eqnarray}
and $g_k \ge 0$ for all $0 \le k \le n$, then we say that
$f(x)$ is {\it semi-$\gamma$-positive}.
\end{definition}

Corresponding to $f(x)$, define a polynomial $g(x)$ by
\begin{eqnarray}\label{sem+g}
g(x)=\sum_{k=0}^{n}g_k x^k.
\end{eqnarray}

In order to show the $\gamma$-positivity of $f(x)$, it is a useful
approach to verifying whether all zeros of $g(x)$ are nonpositive.
The reason is from the next result (see \cite[Remark 3.1.1]{Gal05})
\begin{prop}
Let $f(x)\in \mathbb{R}[x]$ with symmetric coefficients. Then $f(x)$
has nonnegative coefficients and only real zeros if and only if so
does $g(x)$.
\end{prop}

For a symmetric polynomial with nonnegative coefficients, in analogy
to this relation between $\gamma$-positivity and real-rootedness, we
give a criterion for semi-$\gamma$-positivity and Hurwitz stability
as follows.

\begin{thm}\label{thm+sem+gamma}
Let $f(x)$ and $g(x)$ be defined as \eqref{sem+des} and
\eqref{sem+g}, respectively. Then $f(x)$ is Hurwitz stable if and
only if so is $g(x)$. In particular, if $f(x)$ is Hurwitz stable and
its leading coefficient is positive, then $f(x)$ is semi-$\gamma$-positive.
\end{thm}

\begin{proof}
By \eqref{sem+des}, we have
\begin{eqnarray*}
      f(x)
  &=& (1+x)^{\nu}(1+x^2)^{n}g\left(\frac{x}{1+x^2}\right)
      =(1+x)^{\nu}(1+x^2)^{n}g\left(\frac{1}{x+\frac{1}{x}}\right).
\end{eqnarray*}
Let $z =\frac{1}{x+\frac{1}{x}}. $ Obviously, $\Re(z)\Re(x)>0$. In
consequence, we immediately get that the Hurwitz stability of $f(x)$
is equivalent to that of $g(x)$.

In particular, if $f(x)$ is Hurwitz stable and its leading
coefficient is positive, then so is $g(x)$. Thus $g_k \ge 0$ for all
$k $. That is to say that $f(x)$ is semi-$\gamma$-positive.
\end{proof}

Generally speaking, $\gamma$-positivity is stronger than
semi-$\gamma$-positivity. Thus, for a symmetric polynomial $f(x)$,
it may have the semi-$\gamma$-positivity when it is not
$\gamma$-positive. Some such examples will be arranged as follows.

\subsection{Alternating runs of Stirling permutations}

For the generating function $\mathscr{R}_n(x)$ in
\eqref{rec+gf+alt+run+Stir+Perm} of the number of a dual set of
Stirling permutations of order $n$ with $k$ alternating runs, Ma et
al. \cite[Theorem 19]{MMY20} proved the following result concerning
semi-$\gamma$-positivity by using context-free grammars. Obviously,
it is immediate from our Proposition \ref{prop+MW} and Theorem
\ref{thm+sem+gamma}.

\begin{coro}
The polynomial $\mathscr{R}_n(x)$ is semi-$\gamma$-positive.
\end{coro}

\subsection{A class of symmetric polynomials}

Recall \eqref{Rec+main} as follows:
\begin{equation*}
  T_{n+1}(x) = (\bm{\alpha} x^2+\bm{\beta} x+\bm{\gamma})T_{n}(x)
               +(\bm{\mu} x^3+\bm{\nu} x^2+\bm{\varphi} x+\bm{\psi} )D_xT_{n}(x).
\end{equation*}
It is nature to consider the question when is the polynomial
$T_n(x)$ symmetric. Note the fact that a symmetric polynomial $f(x)$
with degree $n$ has the following relation:
\begin{eqnarray}\label{sym+fact}
x^{n+1}D_xf\left(\frac{1}{x}\right)=-nf(x)+xD_xf(x).
\end{eqnarray}
With the help of \eqref{sym+fact}, we obtain a class of symmetric
polynomial $T_n(x)$ satisfying
\begin{eqnarray}\label{Rec+sym}
T_{n+1}(x)=(-m_n\bm{\mu} x^2+\bm{\beta} x+\bm{\beta}+m_n\bm{\nu})T_{n}(x)
           +(\bm{\mu} x^3+\bm{\nu} x^2-\bm{\nu} x-\bm{\mu})D_xT_n(x),
\end{eqnarray}
where $\deg(T_n(x))=m_n$ and $\deg(T_{n+1}(x))=\deg(T_{n}(x))+1$.
In the subsection, we assume $\bm{\mu}\le0\le\bm{\nu}$.

In what follows, we will prove that $T_n(x)$ in (\ref{Rec+sym}) is
Hurwitz stable and semi-$\gamma$-positive. Before it, we need one
criterion for real stability of polynomials.

For multivariate polynomials with real coefficients of degree at
most one, Br\"and\'en gave a criterion about their real stability
(see \cite[Theorem 5.6]{Bra07}). Furthermore, Leake \cite{Lea19}
extended it to general polynomials with real coefficients by Walsh's
coincidence Theorem (see \cite[Theorem 3.4.1.b]{RS02}). We state it
as follows.

\begin{lem}\label{lem+peter}
Let $f \in \mathbb{R}^k[X]$. Then $f$ is real stable if and only if
for all $ i \neq j$ we have
$$
    \Delta_{x_ix_j}=D_{x_i}f \cdot D_{x_j}f-f \cdot D_{x_i}D_{x_j}f \ge 0
$$
and for all $i$ we have
$$
   \Delta_{x_ix_i}=(1-k_i^{-1})(D_{x_i}f)^2-f \cdot D_{x_i}^2 f \ge 0
$$
everywhere in $\mathbb{R}^n$, where $k_i$ is the degree of $x_i$ in
$f$.
\end{lem}

Now, we give the result for Hurwitz stability of $T_n(x)$ as follows.

\begin{thm}\label{thm+sym+HS}
Let $T_n(x)$ be defined by \eqref{Rec+sym} and $\deg(T_n(x))=m_n$.
Assume that $T_0(x)$ is Hurwitz stable. If $\bm{\mu}+\bm{\nu}\le0$
and $2\bm{\beta}+m_n(\bm{\mu}+\bm{\nu})\ge0$, then $T_n(x)$ is Hurwitz
stable and semi-$\gamma$-positive.
\end{thm}

\begin{proof}

We will prove that $T_{n}(x)$ is Hurwitz stable by
induction on $n$, by Theorem \ref{thm+sem+gamma}, which implies that
$T_{n}(x)$ is semi-$\gamma$-positive. Let
$$
T := (-m_n\bm{\mu} x^2+\bm{\beta} x+\bm{\beta}+m_n\bm{\nu})I+(\bm{\mu} x^3+\bm{\nu} x^2-\bm{\nu} x-\bm{\mu})D_x.
$$
We only need to prove that $T$ preserves Hurwitz stability. By
Theorem \ref{thm+HS+peter}, it is equivalent to prove that the
following polynomial
\begin{eqnarray*}\label{symbol}
  & & T(1+xy)^{m_n}                                                                                 \\
  &=& (1+xy)^{m_n-1}[(-m_n\bm{\mu} x^2+\bm{\beta} x+\bm{\beta}+m_n\bm{\nu})(1+xy)
      +m_n(\bm{\mu} x^3+\bm{\nu} x^2-\bm{\nu} x-\bm{\mu})y]                                         \\
  &=& -m_n\bm{\mu}(1+xy)^{m_n-1}\left\{\left(x^2-\frac{\bm{\beta}}{m_n\bm{\mu}}x
      -\frac{\bm{\beta}}{m_n\bm{\mu}}-\frac{\bm{\nu}}{\bm{\mu}}\right)(1+xy)
      +(1-x)\left[1+(1+\frac{\bm{\nu}}{\bm{\mu}})x+x^2\right]y\right\}                              \\
  &=& \left\{\left[(1-x)^2
      +\frac{3\bm{\mu}+\bm{\nu}}{\bm{\mu}-\bm{\nu}}(1+x)^2\right](1+y)
      -\frac{4[\bm{\beta}+m_n(\bm{\mu}+\bm{\nu})]}{m_n(\bm{\mu}-\bm{\nu})}(1+x)(1+xy)\right\}\times \\
  & & \frac{m_n(\bm{\nu}-\bm{\mu})(1+xy)^{m_n-1}}{4}
\end{eqnarray*}
is Hurwitz stable. This is immediate from the next claim.
\begin{cl} For any $r \ge 1 $ and $r+s \ge 1$, the bivariate polynomial
$$
   [(1-x)^2+r(1+x)^2](1+y)+s(1+x)(1+xy)
$$
is Hurwitz stable.
\end{cl}

\begin{proof}

Let $x=-ix, y=-iy$. That is equivalent to show that the right hand
side of the below equality
\begin{eqnarray*}
 & & [(1+ix)^2+r(1-ix)^2](1-iy)+s(1-ix)(1-xy)\\
 &=& (r+1)(1-x^2)-2(r-1)xy+s(1-xy)-[(r-3)x+(r+1)y+(r+s+1)x(1-xy)]i
\end{eqnarray*}
is stable. By Theorem \ref{thm+HB}, it is enough to prove that
\begin{eqnarray*}
(r+1)(1-x^2)-2(r-1)xy+s(1-xy)-[(r-3)x+(r+1)y+(r+s+1)x(1-xy)]z
\end{eqnarray*}
is real stable. By computing, we get
\begin{eqnarray*}
  \Delta_{xy} &=& s(r+s+1)(1-xz)^2+s(r+1)(x-z)^2                               \\
              & & +2(r-1)(r+s+1)(1+x^2z^2)+2(r^2-1)(x^2+z^2) \ge 0,            \\
  \Delta_{xz} &=& s(r+s+1)(1-xy)^2+s(r+1)(x-y)^2                               \\
              & & +2(r-1)(r+s+1)(1+x^2y^2)+2(r^2-1)(x^2+y^2) \ge 0,            \\
  \Delta_{yz} &=& (r+1)(r+s+1)(1-x^2)^2+4(r-1)(r+s-1)x^2     \ge 0,            \\
  \Delta_{xx} &=& 2(2r+s-2)^2(y^2+z^2)+8(r+1)(r+s+1)(1+y^2z^2)-4(s^2+8s+16r)yz \\
              &\ge&  32(r-1)(r+s-1)|yz| \\
              &\ge& 0
\end{eqnarray*}
for any $x, y, z \in \mathbb{R}$ and $r\ge1, r+s\ge1$. Hence,
according to Lemma \ref{lem+peter}, which confirms the claim.
\end{proof}
Therefore, we complete the proof.
\end{proof}

A polynomial sequence $(f_n(q))_{n\ge0}$ is called {\it strongly
$q$-log-convex} if
$$
f_{n+1}(q)f_{m-1}(q)-f_n(q)f_m(q)
$$
has only nonnegative coefficients for any $n \ge m \ge 1$. See
\cite{CWY11, Zhu13} for the details concerning the development of strong
$q$-log-convexity. In the following context, we assume $\delta \in \mathbb{N}^+$.

\begin{thm}\label{thm+symm+T}
Let $T_n(x)$ be defined by \eqref{Rec+sym}, where all $\beta, \mu, \nu$ are
real numbers and $m_n=n-\delta+1$. If $\mu+\nu=0$,
then we have
\begin{itemize}
  \item [\rm (i)]
  its exponential generating function is
  $$
    \sum\limits_{n\ge0}T_{n+\delta-1}(x)\frac{t^n}{n!}
  =
  \frac{(1-x)^{\beta/\nu}}{[(1-x)\cos(\nu(1-x)t)-(1+x)\sin(\nu(1-x)t)]^{\beta/\nu}};
  $$
  \item [\rm (ii)]
  its ordinary generating function
  has the Jacobi continued fraction expansion
    \begin{eqnarray*}
      \sum\limits_{n=0}^{\infty}T_{n+\delta-1}(x) t^n=\DF{1}{1- r_0t-
      \DF{s_1t^2}{1-r_1t-\DF{s_2t^2}{1- r_2t-\ldots}}},
    \end{eqnarray*}
  where $r_i=(2\nu i+\beta)(1+x)$ and $s_i=2\nu i[\beta+\nu(i-1)](1+x^2)$
  for $i\geq0$;
  \item [\rm (iii)]
  the polynomial sequence $(T_n(q))_{n\ge0}$ is strongly
  $q$-log-convex for $\beta\ge 0$;
  \item [\rm (iv)]
  the polynomial $T_n(x)$ is not $\gamma$-positive for
  $n\ge \delta+2$ and $\beta\ge 0$.
\end{itemize}

\end{thm}

\begin{proof}

For (i), define a polynomial $g_n(x)$ for $n\ge0$ by the following relation:
\begin{eqnarray}\label{def+g+T}
g_n(x):=\frac{\delta^n}{2^n}(1+x)^nT_{n+\delta-1}\left(\frac{x-1}{x+1}\right).
\end{eqnarray}

By \eqref{Rec+sym}, then we have a recurrence relation for $g_n(x)$
as follows:
\begin{eqnarray}\label{Rec+g}
g_{n+1}(x)=\beta\delta xg_n(x)+\nu\delta(1+x^2)D_xg_n(x).
\end{eqnarray}
Our aim is to get the exponential generating function of $g_n(x)$.
We first have the following general result.

\begin{cl}\label{cl+expon+f}
Let $\{r,s\}\subseteq\mathbb{R}$ and $\{u,v\}\subseteq\mathbb{R}^{\ge0}$.
Assume that a polynomial sequence $(f_n(x))_{n\geq0}$ satisfies
the following recurrence relation:
\begin{eqnarray}\label{Rec+f}
f_{n+1}(x)=rsxf_n(x)-s(u+vx^2)D_xf_n(x),
\end{eqnarray}
where $f_0(x)=1$. Then the exponential generating function of
$f_n(x)$ is
$$
   \sum\limits_{n\ge0}f_{n}(x)\frac{t^n}{n!}
 = \left[\cos(s\sqrt{uv}t)+\sqrt{v/u}x\sin(s\sqrt{uv}t)\right]^{(r/v)}.
$$
\end{cl}

\begin{proof}

Let the exponential generating function
$$
\mathcal{F}(x,t) :=\sum\limits_{n\ge0}f_{n}(x)\frac{t^n}{n!}.
$$
Then, by \eqref{Rec+f}, we have the next partial differential
equation:
\begin{eqnarray}\label{Rec+F+pde}
\mathcal{F}_t(x,t)=rsx\mathcal{F}(x,t)-s(u+vx^2)\mathcal{F}_x(x,t)
\end{eqnarray}
with the initial condition $\mathcal{F}(x,0)=1$. It is routine
to check that
$$
   \mathcal{F}(x,t)
 = \left[\cos(st\sqrt{uv})+\sqrt{v/u}x\sin(st\sqrt{uv})\right]^{(r/v)}
$$
is a solution of \eqref{Rec+F+pde} with the initial condition.
\end{proof}
In consequence, taking $r=-\beta\delta, s=-1$ and $u=v=\nu\delta$
in \eqref{Rec+f}, then we have the exponential generating
function of $g_n(x)$:
\begin{eqnarray}\label{egf+g}
   \sum\limits_{n\ge0}g_{n}(x)\frac{t^n}{n!}
 = \frac{1}{[\cos(\nu \delta t)-x\sin(\nu \delta t)]^{\beta/\nu}}.
\end{eqnarray}
In addition, it follows from (\ref{def+g+T}) that we have
\begin{eqnarray}\label{Rel+Q+T}
T_{n+\delta-1}(x)=\frac{(1-x)^n}{\delta^n}g_{n}\left(\frac{1+x}{1-x}\right).
\end{eqnarray}
Combining \eqref{egf+g} and \eqref{Rel+Q+T} gives (i).

For (ii), if let
\begin{eqnarray}\label{Rel+T+h}
T_{n+\delta-1}(x)=(1+x)^{n}h_n\left(\frac{-2x}{(1+x)^2}\right)
\end{eqnarray}
for $n\ge0$, then combining \eqref{Rec+sym} and \eqref{Rel+T+h}
derives the recurrence relation of $h_n(x)$ as follows:
\begin{eqnarray*}\label{Rec+h}
   h_{n+1}(x)
 = [2n\nu(x+1)+\beta]h_{n}(x)-2\nu(x+1)(2x+1)D_xh_n(x).
\end{eqnarray*}
Let
\begin{eqnarray}\label{Rel+h+S}
S_n(x)=h_{n}(x-1),
\end{eqnarray}
where $\deg(S_{n}(x))=\lfloor n/2 \rfloor$. Then $S_n(x)$ satisfies
the following recurrence relation:
\begin{eqnarray*}
   S_{n}(x)
 = [2(n-1)\nu x+\beta]S_{n-1}(x)+2\nu x(1-2x)D_xS_{n-1}(x).
\end{eqnarray*}
That is to say, the coefficients $S_{n,k}$ of $S_n(x)$ satisfy
\begin{eqnarray}\label{Rec+S}
S_{n,k}=(2\nu k+\beta)S_{n-1,k}+2\nu(n-2k+1)S_{n-1,k-1},
\end{eqnarray}
where $S_{n,k}=0$ unless $0\le k \le n$ with initial conditions
$S_{0,0}=1$. Then by \cite[(4.10)]{Zhu20aam} we have the Jacobi
continued fraction expansion
\begin{eqnarray}\label{JCF+S}
\sum\limits_{n=0}^{\infty}S_n(x)t^n
=\DF{1}{1- r_0t-\DF{s_1t^2}{1-r_1t-\DF{s_2t^2}{1- r_2t-\ldots}}},
\end{eqnarray}
where $r_i=2\nu i+\beta$ and
$s_i=2\nu i[\nu(i-1)+\beta]x$ for $i\geq0$.

Then by taking $x \rightarrow 1+x$ in \eqref{JCF+S}, we get
\begin{eqnarray}\label{JCF+h}
\sum\limits_{n=0}^{\infty}h_{n}(x)t^n
=\DF{1}{1- r_0t-\DF{s_1t^2}{1-r_1t-\DF{s_2t^2}{1- r_2t-\ldots}}},
\end{eqnarray}
where $r_i=2\nu i+\beta$ and $s_i=2\nu i[\nu(i-1)+\beta](1+x)$ for
$i\geq0$. Moreover, by taking $t \rightarrow (1+x)t$ and $x
\rightarrow \frac{-2x}{(1+x)^2}$ in \eqref{JCF+h}, then we get
\begin{eqnarray*}
\sum\limits_{n=0}^{\infty}T_{n+\delta-1}(x)t^n
=\DF{1}{1- r_0t-\DF{s_1t^2}{1-r_1t-\DF{s_2t^2}{1- r_2t-\ldots}}},
\end{eqnarray*}
where $r_i=(2\nu i+\beta)(1+x)$ and $s_i=2\nu
i[\nu(i-1)+\beta](1+x^2)$ for $i\geq0$.

For (iii), note the following criterion for the strong
$q$-log-convexity \cite{Zhu13}:

\noindent Let
\begin{eqnarray*}
\sum\limits_{n=0}^{\infty}F_{n}(q) t^n=\DF{1}{1-
r_0(q)t-\DF{s_1(q)t^2}{1- r_1(q)t-\DF{s_2(q)t^2}{1-
r_2(q)t-\ldots}}},
\end{eqnarray*}
where both $r_n(q)$ and $s_{n+1}(q)$ are polynomials with
nonnegative coefficients for $n \ge 0$. If all coefficients of
$r_i(q)r_{i+1}(q)-s_{i+1}(q)$ are nonnegative for all $i \ge 0$,
then $(F_n(q))_{n\ge0}$ is strongly $q$-log-convex. For $T_n(q)$, it
is obvious that
\begin{eqnarray*}
&&r_i(q)r_{i+1}(q)-s_{i+1}(q)\\
&= & (2\nu i+\beta)(2\nu i+2\nu+\beta)(1+q)^2-2\nu(i+1)(\nu i+\beta)(1+q^2)         \\
&=& [2\nu^2i^2+2\nu(\nu+\beta)i+\beta^2](1+q^2)+2(2\nu
i+\beta)[2\nu(i+1)+\beta]q
\end{eqnarray*}
has only nonnegative coefficients for $i, \beta\ge0$. Hence
$(T_n(q))_{n\ge0}$ is strongly $q$-log-convex for $\beta\ge 0$.

For (iv), by \eqref{Rel+h+S}, we have
$$
h_{n,k}=\sum\limits_{i\ge0}S_{n,i}\binom{i}{k},
$$
where $S_{n,i}$ satisfies the recurrence relation \eqref{Rec+S}. In
addition, by \eqref{Rec+S}, it is easy to prove that $S_{n,i}$ is
nonnegative for $\beta\ge0$ and $\nu>0$. In consequence, we obtain
the expansion of $T_n(x)$ in the gamma basis
$$
\left\{x^k(1+x)^{n-\delta+1-2k}|0\le
k\le\left\lfloor\frac{n-\delta+1}{2}\right\rfloor\right\}
$$
as follows:
\begin{eqnarray*}
    T_{n}(x)
&=& (1+x)^{n-\delta+1}h_{n-\delta+1}\left(-\frac{2x}{(1+x)^2}\right)       \\
&=& \sum\limits_{k\ge0}h_{n-\delta+1,k}(-2)^kx^k(1+x)^{n-\delta+1-2k}      \\
&=& \sum\limits_{k\ge0}(-2)^k\left(\sum\limits_{i\ge0}S_{n-\delta+1,i}
    \binom{i}{k}\right)x^k(1+x)^{n-\delta+1-2k}.
\end{eqnarray*}
Then, the result is desired. This completes the proof.
\end{proof}

\begin{re}
\em Flajolet \cite{Fla80} gave a general combinatorial
interpretation in terms of weighted Motzkin paths for a
Jacobi continued fraction expansion. From this, by
\eqref{JCF+S}, we can also obtain that $S_n(x)$ has only
nonnegative coefficients in $x$ for $\bm{\beta}\ge0$ and $\bm{\nu}>0$,
see \cite[Remark 5.6]{Zhu21} for instance .
\end{re}

It is known that $\gamma$-positivity is stronger than unimodality.
Though Theorem \ref{thm+symm+T} (iv) says that $T_n(x)$ is not
$\gamma$-positivity, it may still be unimodal.

\begin{thm}\label{thm+T+unimodal}
If $\bm{\beta}=1$ and $\bm{\nu}=-\bm{\mu}=1/{\delta}$, then $T_{n}(x)$ be
defined by \eqref{Rec+sym}  with $m_n=n-\delta+1$ is unimodal for any
$n \ge \delta+2$.
\end{thm}

\begin{proof}

We will prove it by induction on $n$. Whenever $n=\delta+2$, we have
$$
   \delta T_{\delta+2}(x)
 = (4+6\delta+\delta^2)(1+x^3)+(4+6\delta+3\delta^2)(x+x^2),
$$
which is unimodal. Assume that $T_{i}(x)$ is unimodal for
$i=n>\delta+2\ge3$. By induction hypothesis, whenever $i=n+1$, we
need to verify $\delta (T_{n+1, k}- T_{n+1, k-1})\ge0$ for $1 \le k
\le \lfloor (n-\delta+2)/2 \rfloor$.

By \eqref{Rec+sym}, the coefficients of $T_{n+1}(x)$ satisfy the
recurrence relation
\begin{eqnarray*}
     \delta T_{n+1,k}
 &=& (k+1)T_{n,k+1}+(n-k+1)T_{n,k}+(k+\delta-1)T_{n,k-1}
     +(n-\delta-k+3)T_{n,k-2}.
\end{eqnarray*}
It helps us to get
\begin{eqnarray}\label{Rec+P+number}
     \delta(T_{n+1,k}-T_{n+1,k-1})
 &=& (k+1)T_{n,k+1}+(\delta-2)T_{n,k}+(n-\delta-2k+3)(T_{n,k}-T_{n,k-1}) \nonumber  \\
 & & +(n-2\delta-2k+5)T_{n,k-2}-(n-\delta-k+4)T_{n,k-3}.
\end{eqnarray}

For $1 \le k < \lfloor (n-\delta+2)/2 \rfloor$, $T_{n, i}$ is
increasing as $i$ from $0$ to $\lfloor(n-\delta+1)/2\rfloor$ by
assumption. Note that the sum of the coefficients in right hand side
of \eqref{Rec+P+number} is $0$, which implies $$\delta (T_{n+1, k}-
T_{n+1, k-1})\geq0.$$



For $k =\lfloor(n-\delta+2)/2\rfloor$, we will consider two cases in
terms of parity of $n-\delta+2$.

Case $1$: $n-\delta+2=2\ell+1$ and $k=\ell$. Then
$T_{n,\ell+1}=T_{n,\ell-1}$ and
\begin{eqnarray*}
       \delta(T_{n+1,\ell}-T_{n+1,\ell-1})
 &=&   (\ell+1)T_{n,\ell+1}+\delta T_{n,\ell}-2T_{n,\ell-1}
       -(\delta-4)T_{n,\ell-2}-(\ell+3)T_{n,\ell-3}       \\
 &=&   \delta T_{n,\ell}+(\ell-1)T_{n,\ell-1}
       -(\delta-4)T_{n,\ell-2}-(\ell+3)T_{n,\ell-3}       \\
 &\ge& \delta T_{n,\ell-2}+(\ell-1)T_{n,\ell-2}
       -(\delta-4)T_{n,\ell-2}-(\ell+3)T_{n,\ell-2}       \\
 &=&  0
\end{eqnarray*}
because $T_{n, i}$ is increasing as $i$ from $0$ to $\ell$.

Case $2$:  $n-\delta+2=2\ell$ and $k=\ell$. Then $\ell\geq3$, $T_{n,
\ell+1}=T_{n, \ell-2}$ and $T_{n, \ell}=T_{n, \ell-1}$. Thus we have
\begin{eqnarray}\label{eq+dif+ell}
      \delta(T_{n+1,\ell}-T_{n+1,\ell-1})
  &=& (\ell+1)T_{n,\ell+1}+(\delta-1)T_{n,\ell}-T_{n,\ell-1}
      -(\delta-3)T_{n,\ell-2}-(\ell+2)T_{n,\ell-3}      \nonumber  \\
  &=& (\delta-2)T_{n,\ell-1}+(\ell-\delta+4)T_{n,\ell-2}
      -(\ell+2)T_{n,\ell-3}.
\end{eqnarray}

If $\delta\ge2$, then
\begin{eqnarray*}
        \delta(T_{n+1,\ell}-T_{n+1 \ell-1})
  &=&   (\delta-2)T_{n,\ell-1}+(\ell-\delta+4)T_{n,\ell-2}
         -(\ell+2)T_{n,\ell-3}                                     \\
  &\ge& (\delta-2)T_{n,\ell-2}+(\ell-\delta+4)T_{n,\ell-2}
        -(\ell+2)T_{n,\ell-3}                                      \\
  &\ge& (\ell+2)T_{n,\ell-2}-(\ell+2)T_{n,\ell-3}                  \\
  &\ge& 0
\end{eqnarray*}
because $T_{n, i}$ is increasing as $i$ from $0$ to $\ell$.

If $\delta=1$, then (\ref{eq+dif+ell}) becomes to
\begin{eqnarray*}
  & &  T_{n+1,\ell}-T_{n+1,\ell-1}                                 \\
  &=&  (\ell^2+\ell-3)T_{n-1,\ell-1}+(2\ell+8)T_{n-1,\ell-2}
       -(4\ell+11)T_{n-1,\ell-3}                                   \\
  & &  +(6\ell+12)T_{n-1,\ell-4}-(\ell^2+5\ell+6)T_{n-1,\ell-5}    \\
  &=&  (\ell^2-\ell-6)(T_{n-1,\ell-1}-T_{n-1,\ell-5})
       +(2\ell+3)(T_{n-1,\ell-1}-T_{n-1,\ell-3})                   \\
  & &  +(2\ell+8)(T_{n-1,\ell-2}-T_{n-1,\ell-3})
       +(6\ell+12)(T_{n-1,\ell-4}-T_{n-1,\ell-5})                  \\
  &\ge&0
\end{eqnarray*}
for $\ell\ge3$. This completes all proof.
\end{proof}

\subsection{A relation with the derivative polynomials}

The polynomial $T_n(x)$ has a close relation with the derivative
polynomials. Knuth and Buckholtz \cite{KB67} introduced the
derivative polynomials to compute the tangent and secant numbers,
where the derivative polynomial for secant defined by
$$
D^{n}_{\theta}\sec\theta=\sec\theta \cdot Q_n(\tan\theta).
$$
Based on this, Hoffman \cite{Hof99} studied the exponential generating
functions and the combinatorial interpretation of the coefficients for those
polynomials. In addition, he also studied the Springer and Shanks numbers in
terms of the Eulerian polynomials. Josuat-Verg\`es \cite{JV14} defined the
generalized derivative polynomials for secant as follows:
$$
D^{n}_{\theta}\sec^{\delta}\theta=\sec^{\delta}\theta \cdot Q^{(\delta)}_n(\tan\theta),
$$
where $Q^{(\delta)}_n(x)$ satisfies the following recurrence relation:
\begin{equation}\label{Rec+sec}
Q^{(\delta)}_{n+1}(x)=\delta xQ^{(\delta)}_n(x)+(1+x^2)D_xQ^{(\delta)}_n(x)
\end{equation}
with the initial condition $Q^{(\delta)}_0(x)=1$. For the
generalized derivative polynomials, Josuat-Verg\`es studied the
ordinary (resp., exponential) generating functions in terms of the
Jacobi continued fraction expansion (resp., trigonometric functions).
We refer the reader to \cite{Hof95,Hof99,JV14} and references therein
for more details.


%


Combining \eqref{Rec+g}, \eqref{Rec+sec}, \eqref{Rel+T+Q} and (ii)
of Theorem \ref{thm+symm+T} gives the following result. It not only
gives a relation between $T_n(x)$ and $Q^{(\delta)}_{n}(x)$, but
also implies some properties of $Q^{(\delta)}_{n}(x)$.

\begin{prop}\label{prop+Q}
Let $Q^{(\delta)}_{n}(x)$ be defined by \eqref{Rec+sec}. If
$\bm{\beta}=1$ and $\bm{\nu}=-\bm{\mu}=1/{\delta}$, then
\begin{itemize}
  \item [\rm (i)]
  it has the relation with the derivative polynomial
    \begin{equation}\label{Rel+T+Q}
      Q^{(\delta)}_{n}(x)=\frac{\delta^n(1+x)^n}{2^n}T_{n+\delta-1}\left(\frac{x-1}{x+1}\right);
    \end{equation}
  \item [\rm (ii)]
   its exponential generating function is
  $$
     \sum\limits_{n\ge0}Q^{(\delta)}_{n}(x)\frac{t^n}{n!}
   = \frac{1}{(\cos t-x\sin t)^{\delta}};
  $$
  \item [\rm (iii)]
   its ordinary generating function has the Jacobi continued fraction
   expansion
    \begin{eqnarray*}
     \sum\limits_{n=0}^{\infty}Q^{(\delta)}_{n}(x) t^n=\DF{1}{1- r_0t-
     \DF{s_1t^2}{1-r_1t-\DF{s_2t^2}{1- r_2t-\ldots}}},
    \end{eqnarray*}
   where $r_i=(2i+\delta)x$ and $s_i=i(i+\delta-1)(1+x^2)$
   for $i\geq0$.
  \end{itemize}
\end{prop}

\begin{rem}
The (ii) and (iii) in Proposition \ref{prop+Q} were also proved by
Josuat-Verg\`esit \cite{JV14} using the different method.
\end{rem}

In addition, we also give a convolutional relation among the
polynomial $T_{n}(x)$ in \eqref{Rec+sym} for different $\delta$.
For convenience, denote $T^{(\delta)}_{n}(x)=T_n(x)$ for $m_n=n-\delta+1$.
Then, we have the following result.
\begin{prop}\label{prop+P+inv}
If $\bm{\beta}=1$ and $\bm{\nu}=-\bm{\mu}=1/{\delta}$, then we have
\begin{equation*}\label{inv+P}
    (\delta_1+\delta_2)^nT^{(\delta_1+\delta_2)}_{n+\delta_1+\delta_2-1}(x)
  =  \sum\limits_{k\ge0}\binom{n}{k}\delta_1^{k}\delta_2^{n-k}
     T^{(\delta_1)}_{k+\delta_1-1}(x)T^{(\delta_2)}_{n-k+\delta_2-1}(x)
\end{equation*}
for $\delta_1,\delta_2 \in \mathbb{N}$.
\end{prop}

\begin{proof}

By (ii) of Proposition \ref{prop+Q}, we have the following the relation
\begin{eqnarray}\label{inv+Q}
    Q^{(\delta_1+\delta_2)}_n(x)
  = \sum\limits_{k\ge0}\binom{n}{k}
    Q^{(\delta_1)}_{k}(x)Q^{(\delta_2)}_{n-k}(x).
\end{eqnarray}
Combining \eqref{Rel+T+Q} and \eqref{inv+Q} derives the desired result.

\end{proof}

\begin{re}
\em
In \cite{DZ21}, we also obtain some similar results for
$q$-analog of Theorem \ref{thm+symm+T}, Theorem \ref{thm+T+unimodal}
and Proposition \ref{prop+Q}.
\end{re}



\subsection{Alternating descents of permutations}

The number of alternating descents of a permutation $\pi \in S_n$ is
defined by
$$
  altdes_{A} (\pi)
= |\{2i : \pi(2i) < \pi(2i + 1)\} \cup \{2i + 1 : \pi(2i + 1) > \pi(2i + 2)\}|.
$$
Define the \emph{alternating Eulerian polynomial} $\widehat{A}_n(x)$
as follows:
$$
  \widehat{A}_n(x) = \sum\limits_{\pi \in S_n} x^{altdes_{A}(\pi)}
                   = \sum\limits_{k=0}^{n}\widehat{A}(n,k)x^k,
$$
where $\widehat{A}(n,k)$ is called the {\it alternating Eulerian number}.

In recent years, several authors paid attention to the polynomial
$\widehat{A}_n(x)$. For example, Chebikin \cite{Che08} studied the
exponential generating function. Remmel \cite{Re12} computed a
generating function for the joint distribution of the alternating
descent statistic and the alternating major statistic over $S_n$.
Moreover, Gessel and Zhuang \cite{GZ14} extended some results in
\cite{Che08, Re12} by using noncommutative symmetric functions.  For
$n\ge1$, Ma and Yeh \cite{MY16} gave the explicit formula and the
recurrence relation
\begin{eqnarray*}\label{rec+alt+des}
       2\widehat{A}_{n+1}(x)
 & = & [(n-1)x^2+2x+n+1]\widehat{A}_n(x) + (1-x)(1+x^2)D_x\widehat{A}_n(x)
\end{eqnarray*}
with initial conditions $\widehat{A}_1(x)=1$ and
$\widehat{A}_2(x)=1+x$. We sum up some other known properties for
$\widehat{A}_n(x)$ in the following result, which is immediate from
Theorems \ref{thm+sym+HS}, \ref{thm+symm+T} and \ref{thm+T+unimodal}
by taking $\delta=2$.

\begin{thm}\label{thm+alt+des+A}
Let $\widehat{A}_n(x)$ be the alternating Eulerian polynomial of
type $A$. Then
\begin{itemize}
 \item [\rm (i)]
   it has the relation with derivative polynomials
   \begin{eqnarray*}
    \widehat{A}_{n+1}(x) = \frac{(1-x)^n}{2^n} Q^{(2)}_n(\frac{1+x}{1-x});
   \end{eqnarray*}
 \item [\rm (ii)]
   it is symmetric and unimodal for any $n \in \mathbb{N}$;
 \item [\rm (iii)]
   its exponential generating function is
    \begin{equation}\label{egf+A}
        \sum\limits_{n\ge0}\widehat{A}_{n+1}(x)\frac{t^n}{n!}
      = \frac{(1-x)^2}{[(1-x)\cos((1-x)t/2)-(1+x)\sin((1-x)t/2)]^2};
    \end{equation}
 \item [\rm (iv)]
   its ordinary generating function has the Jacobi continued fraction
   expansion
   \begin{eqnarray*}
    \sum\limits_{n=0}^{\infty}\widehat{A}_{n+1}(x) t^n=\DF{1}{1- r_0t-
    \DF{s_1t^2}{1-r_1t-\DF{s_2t^2}{1- r_2t-\ldots}}},
   \end{eqnarray*}
   where $r_i=(i+1)(1+x)$ and $s_i=i(i+1)(1+x^2)/2$ for $i\geq0$;
 \item [\rm (v)]
   it is strongly $q$-log-convex;
 \item [\rm (vi)]
   it is Hurwitz stable and semi-$\gamma$-positive for $n\geq1$;
 \item [\rm (vii)]
  it has the following decomposition
    \begin{eqnarray*}
    \widehat{A}_{n}(x)=\sum_{k\ge0}(-2)^k\left[\sum_{i\ge0}S_{n-1,i}\binom{i}{k}\right]x^k(1+x)^{n-1-2k},
\end{eqnarray*}
   see $S_{n,i}$ in \cite[A094503, A113897]{Slo}.
   Moreover, it is not $\gamma$ positive for $n\geq3$.
\end{itemize}
\end{thm}

\begin{re}
\em We refer the reader to \cite{LMWW20,MY16,Zhu18sdm} for the
corresponding different proof for Theorem \ref{thm+alt+des+A}.
Integrating with respect to \eqref{egf+A} in $t$, we recover the
exponential generating function of $\widehat{A}_{n}(x)$ occurred in
\cite[Theorem 4.2]{Che08} as follows:
\begin{equation*}
   \sum\limits_{n\ge0}\widehat{A}_{n}(x)\frac{t^n}{n!}
 = \frac{\sec(1-x)t+\tan(1-x)t-1}{1-x(\sec(1-x)t+\tan(1-x)t)},
\end{equation*}
since the left part of \eqref{egf+A} is equivalent to 1 whenever $t=0$.
\end{re}

\subsection{Alternating descents of signed permutations}

Similarly, the number of alternating descents of a permutation
$\pi \in B_n$ is defined by
$$
altdes_B(\pi)=|\{2i:\pi(2i)<\pi(2i+1)\}\cup\{2i+1:\pi(2i+1)>\pi(2i+2)\}|,
$$
where $i \ge 0$ and $\pi (0) = 0$. We call $\pi(2i)<\pi(2i+1)$
(resp., $\pi(2i)>\pi(2i+1)$) be the even alternating descent
(resp., ascent) space and $\pi(2i+1)>\pi(2i+2)$ (resp.,
$\pi(2i+1)<\pi(2i+2)$) be the odd alternating descent
(resp., ascent) space. Define the {\it alternating Eulerian polynomial}
of type $B$ be
$$
\widehat{B}_n(x) = \sum\limits_{\pi \in B_n} x^{altdes_B(\pi)}
                 = \sum\limits_{k=0}^{n}\widehat{B}(n,k)x^k,
$$
where $\widehat{B}(n,k)$ is called the {\it alternating Eulerian
number} of type $B$.

We list the first few terms as follows:
$$
\begin{array}{lc}
  \widehat{B}_0(x) = 1,                          \\
  \widehat{B}_1(x) = 1+x,                        \\
  \widehat{B}_2(x) = 3+2x+3x^2,                  \\
  \widehat{B}_3(x) = 11+13x+13x^2+11x^3,         \\
  \widehat{B}_4(x) = 57+76x+118x^2+76x^3+57x^4.  \\
\end{array}
$$
Using the similar combinatorial interpretation of alternating
descent numbers of type $A$ \cite{LMWW20}, we have the next
recurrence relation
\begin{eqnarray*}
 \widehat{B}_{n+1, k}
        = (n-k+2)\widehat{B}_{n,k-2}+k\widehat{B}_{n,k-1}
          +(n-k+1)\widehat{B}_{n,k}+(k+1)\widehat{B}_{n,k+1},
\end{eqnarray*}
which implies
\begin{eqnarray*}
\widehat{B}_{n+1}(x)
        = (nx^2+x+n+1)\widehat{B}_n(x)+(1-x)(1+x^2)D_x\widehat{B}_n(x).
\end{eqnarray*}
We refer the reader to \cite{MFMY21} for a different proof from the
context-free grammar.

In analog to $\widehat{A}_n(x)$, we sum up the other properties of
$\widehat{B}_n(x)$ as follows, which is immediate by Theorems
\ref{thm+sym+HS}, \ref{thm+symm+T} and \ref{thm+T+unimodal} with
$\delta=1$.
\begin{thm}\label{thm+alt+des+B}
Let $\widehat{B}_n(x)$ be the alternating Eulerian polynomial of
type $B$. Then
\begin{itemize}
  \item [\rm (i)]
   it has the relation with derivative polynomials
   \begin{eqnarray*}
    \widehat{B}_n(x)=(1-x)^nQ^{(1)}_n(\frac{1+x}{1-x});
   \end{eqnarray*}
  \item [\rm (ii)]
   it is symmetric and unimodal for any $n \ge 3$;
  \item [\rm (iii)]
   its exponential generating function is
    \begin{eqnarray*}
    \sum\limits_{n\ge0}\widehat{B}_n(x)\frac{t^n}{n!}
    = \frac{1-x}{(1-x)\cos(1-x)t-(1+x)\sin(1-x)t};
   \end{eqnarray*}
  \item [\rm (iv)]
   its ordinary generating function has the Jacobi continued fraction
   expansion
   \begin{eqnarray*}\label{JCF+alt+des+B}
    \sum\limits_{n=0}^{\infty}\widehat{B}_{n}(x) t^n=\DF{1}{1- r_0t-
    \DF{s_1t^2}{1-r_1t-\DF{s_2t^2}{1- r_2t-\ldots}}},
   \end{eqnarray*}
   where $r_i=(2i+1)(1+x)$ and $s_i=2i^2(1+x^2)$ for $i\geq0$;
  \item [\rm (v)]
   the polynomial sequence $(\widehat{B}_n(q))_{n\ge0}$ is strongly
   $q$-log-convex;
  \item [\rm (vi)]
   it is Hurwitz stable and semi-$\gamma$-positive for $n\geq1$;
  \item [\rm (vii)]
   it has the following decomposition
   \begin{eqnarray*}
     \widehat{B}_{n}(x)
    =\sum_{k\ge0}(-4)^k\left[\sum_{i\ge0}\widetilde{W}_{n,i}\binom{i}{k}\right]x^k(1+x)^{n-2k},
   \end{eqnarray*}
   where $\widetilde{W}_{n,i}$ is the left peaks in \eqref{left+peak}.
   Moreover, it is not $\gamma$ positive for $n\geq2$.
\end{itemize}
\end{thm}

\begin{re}
\em  For (i)-(iii) and (vii) of Theorem \ref{thm+alt+des+B}, they
were recently proved by Ma et al \cite{MFMY21} using the different
method.
\end{re}

In particular, taking $\delta_1=\delta_2=1$ in Proposition
\ref{prop+P+inv}, we get a result for the alternating Eulerian
polynomials of types $A$ and $B$ as follows.
\begin{prop}
The alternating Eulerian polynomials of types $A$ and $B$ have
following relation:
\begin{equation*}
2^n\widehat{A}_{n+1}(x)
  =  \sum\limits_{k\ge0}\binom{n}{k}\widehat{B}_{k}(x)\widehat{B}_{n-k}(x).
\end{equation*}
\end{prop}



\section{The alternatingly increasing property}

Let the polynomial $p = \sum_{k=0}^n p_k x^k \in \mathbb{R}[x]$.
We call $p$ \emph{alternatingly increasing} if the coefficients
of $p$ satisfy
$$
0 \le p_0 \le p_n \le p_1 \le p_{n-1} \le \dots \le p_{\left
\lfloor(n+1)/2 \right \rfloor}.
$$
It is obvious that the alternatingly increasing property implies
unimodality, that is to say, it is an approach to proving unimodality
of combinatorial sequences. The unimodality problems have been
extensively investigated in many branches of mathematics, see
\cite{Bra15, Bre94, Sta89} for details concerning the development
of unimodality.

The alternatingly increasing property of a polynomial $p$ has a
close relation with the \emph{symmetric decomposition} of the polynomial
$p$. It is known that every polynomial $p$ of degree at most $n$ can be
uniquely decomposed as $p=a+xb$ where $a$ and $b$ are symmetric with
respect to $n$ and $n-1$, respectively. We call the ordered pair of
polynomial $(a, b)$ the \emph{(symmetric) $\mathcal{I}_n$-decomposition}
of the polynomial $p$. Beck et al. pointed out that a polynomial $p$
is alternatingly increasing if and only if both $a$ and $b$ have only
nonnegative coefficients and are unimodal (see \cite[Lemma 2.1]{BJM19}).

Recently, some authors paid attention to the alternatingly increasing
property that raised combinatorics and geometry. Schepers and Van
Langenhoven \cite{SV13} proved that the coefficients of the
$h^*$-polynomial for a lattice parallelepiped are alternatingly
increasing. Moreover, Beck et al. \cite{BJM19} extended these results
in \cite{SV13} and proved that the $h^*$-polynomial for centrally
symmetric lattice zonotopes and coloop-free lattice zonotopes are
alternatingly increasing. Athanasiadis \cite{At20} proved that
$r$-color Eulerian polynomials, $r$-color derangement polynomials
and binomial Eulerian polynomials are alternatingly increasing by
$\gamma$-positivity decomposition. Br\"and\'en and Solus \cite{BS21}
developed the symmetric decomposition method to prove the
alternatingly increasing property of some polynomials, such as
$r$-color Eulerian polynomials and $r$-color derangement polynomials.
We refer the reader to \cite{At20, BS21, MMY19, SV13} and references
therein for more examples.

In this section, based on the relation between a polynomial and its
reciprocal polynomial, we extend a result of Br\"and\'en and Solus
\cite{BS21}. Therefore, we get the alternatingly increasing property
of some polynomials, such as two kinds of peak polynomials on
$2$-Stirling permutations, descent polynomials on signed
permutations of the $2$-multiset and colored permutations and ascent
polynomials for $k$-ary words. In addition, we also obtain a
recurrence relation and zeros interlacing of the $q$-analog of
descent polynomials on colored permutations that extend some results
of Br\"and\'en and Brenti. Moreover, we get the alternatingly
increasing property of this polynomials. Finally, we show the
alternatingly increasing property and zeros interlacing for two
kinds of peak polynomials on the dual set of Stirling permutations
by using our result for Hurwitz stability.

\subsection{\texorpdfstring{$h$-} ppolynomials}

A polynomial $h(x) \in \mathbb{R}[x]$ is called as
\emph{$h$-polynomial} if it satisfies the following relation:
\begin{equation}\label{Ehrhart+fh}
  \sum\limits_{m \in \mathbb{N}}i(m)x^m=\frac{h(x)}{(1-x)^{n+1}}
\end{equation}
with $i(x) \in \mathbb{R}[x]$ and $\deg(i(x))=n$. And a polynomial
$f(x)$ satisfying the following transformation:
\begin{equation}\label{Rel+fh}
f(h; x)=(1+x)^nh\left(\frac{x}{1+x}\right)
\end{equation}
is called as \emph{$f$-polynomial} of the polynomial $h(x)$ with respect
to $n$. Following the transformation, we know that if $h(x)$ with
nonnegative coefficients has only real zeros, then $f(h; x)$ has all zeros
in $[-1,0]$ and nonnegative coefficients. By \eqref{Rel+fh}, the following
relation is immediate
\begin{equation}\label{Rel+hf}
h(x)=(1-x)^nf\left(\frac{x}{1-x}\right).
\end{equation}
Moreover, if both $h_1(x)$ and
$h_2(x)$ with degree $n$ have only nonnegative coefficients and real zeros,
then we have the following equivalent relation:
$$
h_1(x) \ll h_2(x) \Longleftrightarrow f(h_1;x) \ll f(h_2;x),
$$
which provides a choice to study their properties in an easier way.

For a polynomial $p \in \mathbb{R}[x]$  with degree at most $n$, we
denote
$$
\mathcal{I}_n(p(x)) := x^np(1/x) \quad \text{and} \quad
\mathcal{R}_n(p(x)) :=(-1)^np(-1-x)
$$
Then we know that there exists unique pair
polynomials $\tilde a\in \mathbb{R}[x]$ and $ \tilde b\in
\mathbb{R}[x]$ such that $p=\tilde a+x\tilde b$, where
$\mathcal{R}_n(\tilde a)=\tilde a$ and $\mathcal{R}_{n-1}(\tilde
b)=\tilde b$. We call the ordered pair of polynomials $(\tilde a,
\tilde b)$ the \emph{(symmetric) $\mathcal{R}_n$-decomposition} of
the polynomial $p$. In fact, $\tilde a$ (resp., $\tilde b$) is the
\emph{$f$-polynomial} of $a$ (resp., $b$) for the \emph{(symmetric)
$\mathcal{I}_n$-decomposition} of a polynomial $p$ and
$f(\mathcal{I}_n(p);x)=\mathcal{R}_n(f(p;x))$ by \cite[Lemma 2.3]{BS21}.
Recently, Br\"and\'en and Solus gave several equivalent
forms for the interlacing condition of $a$ and $b$ as follows.

\begin{lem}\cite{BS21}\label{lem+Dec+ab} Let $p \in \mathbb{R}[x]$ have degree
at most $n$ and $\mathcal{I}_n$-decomposition $(a, b)$, for which
both $a$ and $b$
have only nonnegative coefficients. Then the following are equivalent:\\
(1) $b \ll a$,     \\
(2) $a \ll p$,     \\
(3) $b \ll p$,     \\
(4) $\mathcal{I}_n(p) \ll p$.
\end{lem}

Note that $p$ has only nonnegative coefficients and
$\mathcal{I}_n(p)\ll p$, which implies that both $a$ and $b$ have
nonnegative coefficients and interlacing zeros. However, for the
general $\mathcal{I}_n$-decomposition $(a, b),$ the zeros of $a$ do
not interlace those of $b$. Define the \emph{subdivision operator}
$\varepsilon$: $\mathbb{R}[x] \rightarrow \mathbb{R}[x]$ by
$$
\varepsilon\binom{x}{k}=x^k
$$
for all $k \ge 0$, where $\binom{x}{k}=x(x-1) \cdots (x-k+1)/k!$. It
is known that the relation between polynomials $i(x)$ and $h(x)$ in
\eqref{Ehrhart+fh} is $\varepsilon(i(x))=f(h;x)$ by \cite[Lemma
2.7]{BS21}. Thus, the study about \emph{$\mathcal{I}_n$-decomposition}
of $h(x)$ can be transformed to this about
\emph{$\mathcal{R}_n$-decomposition} of $i(x)$.

It is known that the $r$-color Eulerian polynomial $A_{n}^{r}(x)$ have
the following identity relation by Steingr\'imsson \cite{St94}:
\begin{equation}\label{Ehrhart+Ar}
\sum\limits_{m \ge 0}(rm+1)^nx^m=\frac{A_{n}^{r}(x)}{(1-x)^{n+1}}.
\end{equation}
Define a refined polynomial $A_{n,k}^r(x)$ by the relation
$$
\sum\limits_{m
\ge0}(rm)^k(rm+1)^{n-k}x^m=\frac{A_{n,k}^r(x)}{(1-x)^{n+1}}.
$$
Obviously, for $k=0$, $A_{n,0}^r(x)$ is the $r$-colored Eulerian
polynomial of order $n$. Based on this, Br\"and\'en and Solus got
the following general result to show that $A_{n}^{r}(x)$ is
alternatingly increasing for $n \in \mathbb{N}$ and fixed $r \in
\mathbb{N}$.

\begin{thm}\emph{\cite[Theorem 3.1]{BS21}}\label{thm+coloreulerian}
Let a polynomial $p$ be defined by
$$
p =\sum\limits_{r \ge 2}\sum\limits_{k=0}^nc_{r,k}A_{n,k}^r(x)
$$
for some $c_{r,k} \ge 0$. Then $\mathcal{I}_n(p) \ll p$ for
$\deg(p)=n$. In particular, $p$ is real-rooted and
alternatingly increasing.
\end{thm}

Now, we consider a more general situation that $i(x)$ is a
nonnegative combination of some polynomials which have only zeros in
$[-1,0]$. We will give a condition making sure the alternatingly
increasing property of the polynomial $h(x)$. For fixed $k \in
\mathbb{N}^+$, assume $0\le r_{k_1}\le r_{k_2}\le\cdots\le r_{k_{n}}\le 1$
for any $n\in \mathbb{N}$, and we let
\begin{eqnarray}\label{carlitz+h}
\sum\limits_{m\ge0}\prod_{i=1}^{n}(m+r_{k_i})x^m=\frac{h_{n,k}(x)}{(1-x)^{n+1}}.
\end{eqnarray}
The next more general result in particular implies Theorem
\ref{thm+coloreulerian} by taking $r_{k_i}\in \{0,1/r\}$ and
$c_k=r^n$ for $r\ge2$.

\begin{thm}\label{thm+alter+incr}
Let $h_{n,k}(x)$ be defined in (\ref{carlitz+h}). Assume that
$p\in\mathbb{R}[x]$ and has the expression
$$
p=\sum\limits_{k\ge1}c_kh_{n,k}(x)
$$
for all $c_k\ge0$ and the $\mathcal{I}_n$-decomposition $(a, b)$. If
$0\le r_{k_i}+r_{\ell_{n-i+1}}\le1$ for any $k,\ell,i\in\mathbb{N}^+$,
then $\mathcal{I}_n(p)\ll p$ for $\deg(p)=n$. In particular, $b\ll a$
and $p$ is alternatingly increasing.
\end{thm}

\begin{proof}

Let
$$
   i_k(x)
 = \prod_{i=1}^{n}(x+r_{k_i}) \quad\text{and}\quad i(x)
 = \sum\limits_{k\ge1}c_ki_k(x).
$$
Note that both $\varepsilon$ and $\mathcal{R}$ are linear operators,
then
$$
\sum\limits_{k\ge1}c_k\varepsilon(i_k(x))=\varepsilon(i(x))
$$
Taking $\{ \binom{x}{k} \}_{k=0}^{n}$ as a set of basis of $\mathbb{R}[x]_{n}$,
it is easy to verfy that the operator $\mathcal{R}$ and $\varepsilon$
have commutativity on this basis. Thus
$$
   \mathcal{R}_n(\varepsilon(i_k(x)))
 = \varepsilon(\mathcal{R}_n(i_k(x)))
 = \varepsilon(\prod_{i=1}^{n}(x+1-r_{k_i})).
$$

Note that the fact (see \cite[Theorem 4.6]{Bra06}): Assume that two
standard polynomials $f$ and $g$ both have only real zeros $\alpha_n
\le \dots \le \alpha_2 \le \alpha_1$ and $\beta_n \le \dots \le
\beta_2 \le \beta_1$, respectively. If all these zeros are in the
interval [-1, 0] and $\alpha_k \le \beta_k$ for all $k \in [n]$,
then $\varepsilon(f) \ll \varepsilon(g)$.

By the assumption $0\le r_{k_i}+r_{\ell_{n-i+1}}\le1$ for any
$k,\ell,i \in\mathbb{N}^+$ and the above fact, we derive
$\mathcal{R}_n(\varepsilon(i_k(x)))\ll \varepsilon(i_\ell(x))$ for
any $k,\ell\in\mathbb{N}^+$. By Proposition \ref{prop+covn} and
$c_k\ge0$ for any $k \in \mathbb{N}^+$, then we obtain
$\mathcal{R}_n(\varepsilon(i(x)))\ll \varepsilon(i(x))$. In addition,
$p$ has only nonnegative coefficients and real zeros by \eqref{Rel+hf}
since $\varepsilon(i(x))=f(p;x)$.  Combining $\mathcal{R}_n(f(p;x))
=f(\mathcal{I}_n(p);x)$ and $\mathcal{R}_n(\varepsilon(i(x)))\ll
\varepsilon(i(x))$ derives $f(\mathcal{I}_n(p);x) \ll f(p;x)$,
thus $\mathcal{I}_n(p) \ll p$. The alternatingly increasing
property of $p(x)$ and $b\ll a$ are immediate by Lemma
\ref{lem+Dec+ab}.
\end{proof}

\begin{re}
\em
Define the linear map $\mathcal{D}: \mathbb{R}[x] \rightarrow \mathbb{R}[x]$
by
$$
\mathcal{D}(x^k)=d_k(x)
$$
for all $k\ge0$, where $d_k(x)$ is the $k$-th derangement
polynomial. Then, we have an analogous result to Theorem
\ref{thm+alter+incr}. Taking $p=\sum_{k\ge1}c_kh_{n,k}(x)$, where
$h_{n,k}(x)$ is defined by \eqref{carlitz+h}. If $c_k\ge0$ for all
$k\in[n]$, then $\mathcal{D}(p) \ll \mathcal{I}_n(\mathcal{D}(p))$
for $\deg(p)=n$. The proof is similar to Corollary 3.7 in \cite{BS21},
so we omit it here for brevity. In fact, it is more general than
Corollary 3.7 in \cite{BS21}, which can be used to prove
$\mathcal{I}_n(d_{n,r}) \ll d_{n,r}$, where $d_{n,r}$ is the $n$-th
$r$-color derangement polynomial.
\end{re}

\subsection{Ascent polynomials for \texorpdfstring{$k$-} aary words}

Let $w \in S=\{0, 1, \dots, k-1\}^n$ be a $k$-ary words of length
$n$. We assume $w_0 = 0$ for the convention. Let $asc(w)$ denote
the number of $w_i < w_{i+1}$ for $i\in[n-1]\cup\{0\}$. Then
the $n$-th {\it ascent polynomial} for $k$-ary words is defined by
\begin{eqnarray}\label{poly+ary}
\mathscr{A}_{n}^{k}(x)=\sum\limits_{w\in S}x^{asc(w)}.
\end{eqnarray}
It is known that $\mathscr{A}_n^{k}(x)$ has the following relation (see
\cite[Corollary 8]{SS12}):
$$
   \sum\limits_{m\ge0}\binom{n+km}{n}x^m
 = \frac{\mathscr{A}_{n}^{k}(x)}{(1-x)^{n+1}}.
$$
That is to say,
$$
   \sum\limits_{m\ge0}\frac{k^n}{n!}\prod_{i=1}^{n}\left(m+\frac{i}{k}\right)x^m
 = \frac{\mathscr{A}_{n}^{k}(x)}{(1-x)^{n+1}}.
$$
Taking $r_i=i/k$ for $i\in [n]$, $c_1=k^n/n!$ and the others to be
zero in Theorem \ref{thm+alter+incr}, we get the following result.

\begin{prop}
Let the ascent polynomial $\mathscr{A}_{n}^{k}(x)$ be defined by
\eqref{poly+ary} and $(a, b)$ be its $\mathcal{I}_n$-decomposition.
If $k>n$, then $\mathcal{I}_n(\mathscr{A}_{n}^{k})\ll \mathscr{A}_
{n}^{k}$ for $\deg(\mathscr{A}_{n}^{k}(x))=n$. In particular,
$\mathscr{A}_{n}^{k}(x)$ is alternatingly increasing and $b\ll a$.
\end{prop}

\subsection{Descent polynomials on signed permutations of the \texorpdfstring{$2$-} mmultiset}

Recently, Lin \cite{L15} considered the descent polynomials on
signed permutations of the general multiset $M_\textbf{s}
:=\{1^{s_1} , 2^{s_2}, \ldots , n^{s_n}\}$ for each vector
$\textbf{s} := (s_1,s_2, \ldots, s_n)$. Let $s=s_1+s_2+\cdots+s_n$
and $\pi_0=0$. Define $p_{\textbf{s}}^{\pm}(x)$ by
$$
   p_{\textbf{s}}^{\pm}(x)
 = \sum\limits_{\pi \in p_{\textbf{s}}^{\pm}}x^{des \pi},
$$
where $p_{\textbf{s}}^{\pm}$ is the set of all permutations $\pi=\pm
\pi_1 \pm \pi_2 \cdots \pm \pi_s$ with  $\pi_1 \pi_2 \cdots \pi_s$
be a permutation on the multiset $M_\textbf{s}$ and $des\pi$ is the
descent number of $\pi$. Moreover, Lin got the following
relationship:
$$
  \sum\limits_{m \ge 0} \prod_{r=1}^n \frac{(2m+1)(2m+2) \dots (2m+s_r)}{s_j!}x^m
= \frac{p_{\textbf{s}}^{\pm}(x)}{(1-x)^{s+1}}.
$$
In particular, let $p_s(x) = p_{\textbf{s}}^{\pm}(x)$ whenever $s_j \in
\{1,2\}$ for all $j \in [n]$, namely,
\begin{equation}\label{Ehrhart+sign}
  \sum\limits_{m \ge 0} (m+1)^{s-n}(2m+1)^nx^m
= \frac{p_{s}(x)}{(1-x)^{s+1}}.
\end{equation}
%

For the polynomial $p_s(x)$, we have the following result.

\begin{prop}\label{prop+des+sign}
Let $p_s(x)$ satify \eqref{Ehrhart+sign} and $(a,b)$ be its
$\mathcal{I}$-decomposition. Then $\mathcal{I}_{s-1}(p_{s}) \ll
p_s$. In particular, $p_s(x)$ is alternatingly increasing and $b \ll
a$.
\end{prop}

\begin{proof}
Let $i(x)=(x+1)^{s-n}(2x+1)^n$. We obtain
$\varepsilon(i(x))=f(p_s;x)$. In consequence, we have
$$
   \varepsilon(\mathcal{R}_n(i(x)))
 = \mathcal{R}_n(\varepsilon(i(x)))
 = \mathcal{R}_n(f(\mathcal{I}_{s}(\mathcal{I}_{s}(p_s));x))
 = f(\mathcal{I}_{s}(p_s);x).
$$
That is to say, $\mathcal{I}_{s}(p)$ satisfies the following
relation:
$$
   \sum\limits_{m \ge 0}2^nm^{s-n}\left(m+\frac{1}{2}\right)^nx^m
 = \frac{\mathcal{I}_{s}(p_{s})}{(1-x)^{s+1}}.
$$
Taking $r_i\in\{0,1/2\}$, $c_1=2^n$ and the others to be zero in
Theorem \ref{thm+alter+incr}, we have $p_s \ll \mathcal{I}_{s}(p_{s})$.
Note that $\deg(p_s)=s-1$, thus $\mathcal{I}_{s-1}(p_{s}) \ll p_s$.
Both the alternatingly increasing property of $p_s(x)$ and $b \ll a$
are immediate by Theorem \ref{thm+alter+incr}.
\end{proof}

\begin{re}
\em For $s_j=2$ with all $j \in [n]$, the alternatingly increasing
property of $p_s(x)$ was also proved by Ma et al. \cite[Theorem
11]{MMY19} in a different method.
\end{re}

\subsection{Descent polynomials on \texorpdfstring{$r$-} ccolored permutations}

The half Eulerian polynomials of type B are given by
$$
   B_n^+(x)
 = \sum\limits_{\pi \in \mathcal{B}_n^+ } x^{des_B \pi}
   \quad \text{and} \quad B_n^-(x)
 = \sum\limits_{\pi \in \mathcal{B}_n^-}x^{des_B \pi},
$$
where $\mathcal{B}_n^+$ (resp., $\mathcal{B}_n^-$) is the Coxeter
group of type $B$ of rank $n$ with $\pi_n > 0$ (resp., $\pi_n < 0$). By
bijection from $\mathcal{B}_n^+$ to $\mathcal{B}_n^-$, it is easy to
know that $B_n^-(x)= \mathcal{I}_n(B_n^+(x))$ (see \cite[Lemma
7.1]{AC13}) since $\deg(B_n^+(x))=n-1$. And by \cite[(7.5) ]{AC13},
we have
$$
   \sum\limits_{m \ge 0} \left[(2m+1)^n-(2m)^n  \right]x^m
 = \frac{B_n^+(x)}{(1-x)^n},
$$
$$
   \sum\limits_{m \ge 0} \left[(2m)^n-(2m-1)^n  \right]x^m
 = \frac{B_n^-(x)}{(1-x)^n}.
$$

The wreath product group $\mathbb{Z}_r \wr S_n$ consists of all
permutations $\pi \in [0,r-1]\times [n]$. Namely, the element in
$\mathbb{Z}_r \wr S_n$ is thought of as $\pi = \xi^{e_1}\pi_1
\xi^{e_2}\pi_2 \cdots \xi^{e_n}\pi_n$, where $e_i \in [0,r-1]$
and $\pi \in S_n$. Define the following total order relation
on the elements of $\mathbb{Z}_r \wr S_n$:
$$
\xi^{r-1}n<\cdots<\xi n<\cdots<\xi^{r-1}2<\cdots<\xi2
<\xi^{r-1}1<\cdots<\xi1<0<\xi^0 1<\cdots<\xi^0 n.
$$

Assume that $(\mathbb{Z}_r \wr S_n)^+$ is the set of
colored permutations $\pi \in \mathbb{Z}_r \wr S_n$ with
first coordinate of zero color and $des(\pi)$ is the descent number
of $\pi$. Athanasiadis \cite{At20} defined the following polynomial
\begin{equation}\label{r+Eulerian}
   A_{r,n}^+(x)
 = \sum\limits_{\pi \in (\mathbb{Z}_r \wr S_n)^+} x^{des(\pi)},
\end{equation}
The first three terms are listed as follows:
\begin{eqnarray*}
 A_{r,1}^+(x) &=& 1,                               \\
 A_{r,2}^+(x) &=& 1+(2r-1)x,                       \\
 A_{r,3}^+(x) &=& 1+(3r^2+3r-2)x+(3r^2-3r+1)x^2.
\end{eqnarray*}
Athanasiadis showed that $A_{r,n}^+(x)$ can be interpreted as the
$h^*$-polynomial of a lattice polyhedral complex and got the
following expression:
\begin{equation}\label{Ehrhart+Ar+plus}
   \sum\limits_{m \ge 0}\left[ (rm+1)^n -(rm)^n\right]x^m
 = \frac{A_{r,n}^+(x)}{(1-x)^n}.
\end{equation}
Obviously, $A_{r,n}^+(x)$ can be looked as a generalization of
$B_n^+(x)$ because $A_{2,n}^+(x) = B_n^+(x)$. Note that
$\deg(A_{r,n}^+(x))=n-1$, thus we have the following result.

\begin{prop}\label{prop+color+RZ}
Let $A_{r,n}^+(x)$ be defined by \eqref{Ehrhart+Ar+plus}. Then
$\mathcal{I}_{n-1}(A_{r,n}^+)\ll A_{r,n}^+$. In particular,
$A_{r,n}^+(x)$ is alternatingly increasing for $r \ge 2$ and
$n \in \mathbb{N}^+$.
\end{prop}

\begin{proof}

At first, we have the following decomposition:
\begin{eqnarray*}
   (rm+1)^n-(rm)^n
 = \sum\limits_{k = 0}^{n-1}r^{n-1}m^k\left(m+\frac{1}{r}\right)^{n-1-k}.
\end{eqnarray*}
Taking $r_{k_i}\in\{0, 1/r\}$ and $c_k=r^{n-1}$, then the desired
result is immediate by Theorem \ref{thm+alter+incr}.
\end{proof}

Note that we have $B_n^+(x)\ll B_n^-(x)$ whenever $r=2$. It can be
used to prove the real rootedness of the Eulerian polynomials of
type B that was proved by Hyatt \cite{Hya16} using compatible
polynomials and Yang and Zhang \cite{YZ15} in terms of Hurwitz
stability.


\begin{re}
\em Athanasiadis \cite{At20} gave the explanation of $A_{r,n}^+(x)$
by Ehrhart theory. Namely, $(rm+1)^n-(rm)^n$ is equal to the number
of lattice points in the $m$th dilate of the union of the $n$ facets
of $P$ which do not contain the origin, where $P$ is the $r$th
dilate of the standard unit $n$-dimensional cube. Define
$$
A_{r,n}^-(x) = \sum\limits_{\pi \in (\mathbb{Z}_r \wr
\mathcal{S}_n)^-} x^{des(\pi)},
$$
where $(\mathbb{Z}_r \wr S_n)^-$ is the set of colored permutations
$\pi \in \mathbb{Z}_r \wr S_n$ with first coordinate of non-zero color.
By \eqref{Ehrhart+Ar} and \eqref{Ehrhart+Ar+plus}, we can get the
following equality:
\begin{equation}\label{Ehrhart+Ar+minus}
   \sum\limits_{m\ge 0}\left[(rm)^n -(rm-r+1)^n\right]x^m
 = \frac{A_{r,n}^-(x)}{(1-x)^n}.
\end{equation}
We will give an explanation of $A_{r,n}^-(x)$ by Ehrhart theory. Let
$P$ be the $r$th dilate of the standard unit $n$-dimensional cube.
Then $(rm)^n -(rm-r+1)^n$ is equal to the number of lattice points
in the $m$th dilate of the union of the lattice point that is $i\in
[r(r-1)]$ units away from the $n$ facets of $P$ which do not contain
the origin. That is to say, $A_{r,n}^-(x)$ is the $h^*$-polynomial
of a lattice polyhedral complex, namely the collection of all faces
of the facet that is $i\in [r-1]$ units away from $n$ facets of $P$
which do not contain the origin.
\end{re}

In \cite{BL20}, Br\"and\'en and Leander considered the $q$-analog of
the $r$-colored Eulerian polynomials
\begin{eqnarray}\label{def+Euler+color}
    A_{n}^{r}(x;q_1,q_2,\dots,q_n)
 := \sum\limits_{\pi\in \mathbb{Z}_r \wr \mathcal{S}_n}
    x^{des(\pi)}q_1^{e_1(\pi)}q_2^{e_2(\pi)}\cdots q_n^{e_n(\pi)},
\end{eqnarray}
where $e_i(\pi)=e_i$. For example, $\pi=\xi^{1}3 \xi^{3}1
\xi^{0}2\xi^{2}4\xi^{4}4$, the responding term in the polynomial
$A_{n}^{r}(x;q_1,q_2,\dots,q_n)$ is
$x^4q_1^{1}q_2^{3}q_3^{0}q_4^{2}q_5^{4}$. For $r\in\mathbb{N}^+$ and
$q\ge0$, denote $[r]_q:=1+q+q^2+\cdots+q^{r-1}$. We have the
following result.
\begin{prop}\label{prop+Euler+color}
For $n \in \mathbb{N}$ and $r \in \mathbb{N}^+$, let $ A_{n}^{r}
(x;q_1,q_2,\dots,q_n)$ be defined by \eqref{def+Euler+color}. Then
we have
\begin{itemize}
  \item [\rm (i)]
   its recurrence relation is
    \begin{eqnarray}\label{Rec+Eur+q}
          A_{n}^{r}(x;q_1,q_2,\dots,q_n)
     &=& [(n[r]_{q_n}-1)x+1]A_{n-1}^{r}(x; q_1,q_2,\dots,q_{n-1}) \nonumber \\
     & & +[r]_{q_n}x(1-x)D_xA_{n-1}^{r}(x; q_1,q_2,\dots,q_{n-1}),
   \end{eqnarray}
    where $A_{1}^{r}(x;q_1,q_2,\dots,q_n)=([r]_{q_1}-1)x+1$;
  \item [\rm (ii)]
   $A_{n}^{r}(x;q_1,q_2,\dots,q_n) \ll A_{n+1}^{r}(x;q_1,q_2,\dots,q_n)$
   for $q_i\ge0$;
  \item [\rm (iii)]
   $\mathcal{I}_n
   (A_{n}^{r}(x;q_1,q_2,\dots,q_n)) \ll A_{n}^{r}(x;q_1,q_2,\dots,q_n)$ whenever $r\ge2, q_i\ge0$ and $0\le [r]_{q_i}+[r]_{q_{n-i+1}}\le[r]_
   {q_i}[r]_{q_{n-i+1}}$ for any $i\in[n]$;
  \item [\rm (iv)]
   the polynomial $A_{n}^{r}(x;q_1,q_2, \dots, q_n)$ is alternatingly increasing for
   $r\ge2$ and $q_i\ge0$.
\end{itemize}
\end{prop}

\begin{proof}

For (i), Br\"and\'en and Leander in \cite{BL20} used $s$-lecture hall
$P$-partitions to get the following identity
\begin{eqnarray}\label{Rel+Euler+q}
    \sum\limits_{m\ge0}\prod_{i=1}^{n}([r]_{q_i}m+1)x^m
  = \frac{A_{n}^{r}(x;q_1,q_2,\dots,q_n)}{(1-x)^{n+1}}.
\end{eqnarray}
It is easy to check that the recurrence relation \eqref{Rec+Eur+q} satisfies
the identify \eqref{Rel+Euler+q} with initial condition $A_{1}^{r}(x;q_1,q_2,
\dots,q_n)=([r]_{q_1}-1)x+1$, we omit the proceed here.

For (ii), the result is immediate by using the method of zeros
interlacing (see \cite{LW} for details).

For (iii) and (iv), we rewrite \eqref{Rel+Euler+q} as
\begin{eqnarray}
      \sum\limits_{m\ge0}\prod_{i=1}^{n}[r]_{q_i}\prod_{i=1}^{n}
      \left(m+\frac{1}{[r]_{q_i}}\right)x^m
    = \frac{A_{n}^{r}(x;q_1,q_2,\dots,q_n)}{(1-x)^{n+1}}.
\end{eqnarray}
Taking $r_i=1/[r]_{q_i}$, $c_1=\prod_{i=1}^{n}[r]_{q_i}$ and the
others to be zero in Theorem \ref{thm+alter+incr} whenever $r\ge2$
and $q_i\ge0$, we get the desired results.
\end{proof}

\begin{re}
\em
In particular, the polynomial $A_{n}^{r}(x;q_1,q_2,\dots,q_n)$ is the
$q$-analog of Eulerian polynomial type of $B$ whenever $r=2$ and $q_i=q_j$
for $i, j \in [n]$ and is the $r$-colored Eulerian polynomial whenever
$q_i=1$ for $i \in [n]$, whose alternatingly increasing property was
obtained in \cite{BS21}. In addition, Proposition \ref{prop+Euler+color}
can be looked as the further generalization of Theorem 6.4 in \cite{Bra06}
and Theorem 3.4 in \cite{Bre94EJC}.

\end{re}

\subsection{Peak polynomials on dual set of \texorpdfstring{$2$-} SStirling permutations}

Denote $i^j = \underbrace{ i, i, \ldots, i }_j$ for $i, j \ge 1$.
Stirling permutations were defined by Gessel and Stanley
\cite{GS78}. A \emph{Stirling permutation} of order $n$ is a
permutation $\pi$ of the multiset $\{1^2, 2^2, \ldots , n^2\}$ such
that $\pi_s > \pi_k$ for all $k < s < \ell$ whenever $\pi_k = \pi_\ell$.
Moreover, we say that a permutation of the multiset $\{1^r, 2^r,
\ldots, n^r\}$ is a \emph{$r$-Stirling permutation} of order $n$,
denoted as $\mathcal{Q}_{n,r}$, if $\pi_s \ge \pi_k$ for all $k < s
< \ell$ whenever $\pi_k = \pi_\ell$.

In this subsection, we will consider the peak polynomials on the
generalization of $r$-Stirling permutations, which extend the dual
set of $2$-Stirling permutations in \cite{MMYY20}. Let $\pi = \pi_1
\pi_2 \dots \pi_{rn} \in \mathcal{Q}_{n,r}$ and define $\Phi_r$ be
the injection which maps each $\ell$-th occurrence of entry $i$ in
$\pi$ to $ri-\ell+1$. For example, $\Phi_3(111233322)=(321698754)$
whenever $n=3, r=3$. Define the $r$-multiple set
$\Phi_r(\mathcal{Q}_{n,r})$ of $\mathcal{Q}_{n,r}$ as follows:
$$
  \Phi_r(\mathcal{Q}_{n,r})
= \{\pi : \sigma \in \mathcal{Q}_{n,r},\Phi_r(\sigma) = \pi \}.
$$

The statistics \emph{interior peak} and \emph{left peak} in $\pi \in
\mathcal{Q}_{n,r}$ were defined by
\begin{eqnarray*}
  ipk(\pi) &=& |\{ i \in [rn-r+1]\setminus \{1\}:
               \pi_{i-1} < \pi_{i} > \pi_{i+1}> \dots > \pi_{i+r-1} \}|, \\
  lpk(\pi) &=& |\{ i \in [rn-r+1]:
               \pi_{i-1} < \pi_{i} > \pi_{i+1}> \dots > \pi_{i+r-1} \}|,
\end{eqnarray*}
where $\pi_0=0$. Thus we can define the peak polynomials on
$\Phi_r(\mathcal{Q}_{n,r})$ as follows:
$$
M_{n,r}(x)=\sum\limits_{\pi \in
\Phi_r(\mathcal{Q}_{n,r})}x^{ipk{\pi}}, \quad
\widetilde{M}_{n,r}(x)=\sum\limits_{\pi \in
\Phi_r(\mathcal{Q}_{n,r})}x^{lpk{\pi}}.
$$
Let $M_{n,r,k}$ denote the number of $\pi \in \Phi_r(\mathcal{Q}_{n,r})$
with $k$ interior peaks, which can be obtained from $\Phi_r(\mathcal{Q}_{n-1,r})$
by the following two cases:
\begin{itemize}
  \item [\rm (1)]
  For $i \in ipk(\pi)$ and $j\in\{-1,0\}\cup[r-2]$, inserting
  $(rn)(rn-1)\cdots(rn-r+1)$ into the right-hand side of $\pi_{i+j}$
  will preserve the number of $ipk(\pi)$. In addition, inserting
  $(rn)(rn-1)\cdots(rn-r+1)$ into the left-hand side of $\pi_1$ also
  preserves the number of $ipk(\pi)$. Thus, if $ipk(\pi)=k$, then there
  are $rk+1$ ways to obtain a permutation in $\Phi_r(\mathcal{Q}_{n,r})$
  with $k$ interior peaks.
  \item [\rm (2)]
  For $i \notin \{\ell+j: \ell \in ipk(\pi)\quad\&\quad j\in\{-1,0\}\cup[r-2]\}$,
  inserting $(rn)(rn-1)\cdots(rn-r+1)$ into the right-hand
  side of $\pi_{i}$ will increase the number of $ipk(\pi)$
  by $1$. Thus, if $ipk(\pi)=k-1$, then there are $r(n-1)-r(k-1)=r(n-k)$
  ways to obtain a permutation in $\Phi_r(\mathcal{Q}_{n,r})$ with $k$
  interior peaks.
\end{itemize}

Then we can get the following recurrence relation for $M_{n,r,k}$:
\begin{eqnarray}\label{Rec+peak+r}
M_{n,r,k}=(rk+1)M_{n-1,r,k}+r(n-k)M_{n-1,r,k-1}.
\end{eqnarray}
By \eqref{Rec+peak+r}, $M_{n,r}(x)$ satisfies the recurrence
relation:
\begin{equation}\label{poly+ipeak}
\left\{
\begin{array}{lc}
 M_{n,r}(x)=\left[(rn-r)x+1\right]M_{n-1,r}(x)+rx(1-x)D_xM_{n-1,r}(x),    &\\
 M_{1,r}(x)=1, M_{2,r}(x)=1+rx.                                           &\\
\end{array}
\right.
\end{equation}
In fact, $M_{n,r}(x)$ is equivalent to the $1/r$-Eulerian polynomial
$\mathcal{A}_n^r(x)$ because
$$
    \mathcal{A}_n^r(x)
  = [(rn-r)x+1]\mathcal{A}_{n-1}^r(x)+rx(1-x)D_x\mathcal{A}_{n-1}^{r}(x)
$$ with $\mathcal{A}_1^r(x)=1$, see \cite{Bre00,SV12}.

 Similarly,
$\widetilde{M}_{n,r}(x)$ satisfies the recurrence relation:
\begin{equation}\label{poly+lpeak}
\left\{
\begin{array}{lc}
\widetilde{M}_{n,r}(x)=(rn-r+1)x\widetilde{M}_{n-1,r}(x)+rx(1-x)D_x\widetilde{M}_{n-1,r}(x), &\\
\widetilde{M}_{0,r}(x)=1, \widetilde{M}_{1,r}(x)=x.                                     &\\
\end{array}
\right.
\end{equation}

By \eqref{poly+ipeak} and \eqref{poly+lpeak}, we obtain
$M_{n,r}(x)=\mathcal{I}_n(\widetilde{M}_{n,r}(x))$. Obviously,
$\widetilde{M}_{n,r}(x)$ is a special case of the generalized
Eulerian polynomial $\mathscr{T}_{n}(x)$ in (\ref{rec+gene+Eur+gf})
by taking $d=0$ and $\lambda=1$. By Corollary
\ref{Coro+Z+Turan+gen+Eur}, the following result is immediate.


\begin{prop}\label{prop+peak+HS}
Let $(M_{n,r}(x))_{n \ge 0}$ and $(\widetilde{M}_{n,r}(x))_{n \ge
0}$ be defined by \eqref{poly+ipeak} and \eqref{poly+lpeak},
respectively. Then the Tur\'an expressions of $(M_{n,r}(x))_{n\ge
0}$ and $(\widetilde{M}_{n,r}(x))_{n \ge 0}$ are Hurwitz stable for
all $r \ge 2$.
\end{prop}

%


\begin{rem}
Obviously, Proposition \ref{prop+peak+HS} implies that all $(M_{n,r}(q))_{n\ge 0},
(\widetilde{M}_{n,r}(q))_{n \ge 0}$ and $(\mathcal{A}_n^r(q))_{n \ge 0}$
are $q$-log-convex for any $r \ge 2$. In fact, they are all $q$-Stieltjes
moment by Theorem \cite[Theorem 1.3]{Zhu20aam}, i.e., all minors of their
Hankel matrices are polynomials with nonnegative coefficients.
\end{rem}

Constructing a new polynomial sequence $(T_{n,r}(x))_{n \ge 0}$ as follows:
\begin{equation}\label{poly+peak}
(1+x)T_{n,r}(x):=xM_{n,r}(x^2)+\widetilde{M}_{n,r}(x^2).
\end{equation}
By \eqref{poly+ipeak}-\eqref{poly+peak}, we get the recurrence
relation of $T_{n,r}(x)$ as follows:
$$
T_{n+1,r}(x)=\left(rnx^2+\frac{rx-r+2}{2}\right)T_{n,r}(x)+\frac{rx}{2}(1-x^2)D_xT_{n,r}(x)+\frac{r-2}{2}(1-x)\widetilde{M}_{n,r}(x^2).
$$

Based on empirical evidence and computer's arithmetic for
$T_{n,r}(x)$, we propose the following conjecture.

\begin{conj}
Let $T_{n,r}(x)$ satisfy \eqref{poly+peak}. Then $T_{n,r}(x)$ is
Hurwitz stable for all $r \ge 2$ and $n \in \mathbb{N}$.
\end{conj}

Note $M_{n,r}(x)=\mathcal{I}_n(\widetilde{M}_{n,r}(x))$. Thus if
this conjecture is true, then it implies that both $M_{n,r}(x)$ and
$\widetilde{M}_{n,r}(x)$ are alternatingly increasing for all $r\ge2
$ and $n\in \mathbb{N}$. In the following, we will prove this
conjecture for $r=2$. Before it, we need a criterion for two
zeros-interlacing polynomials.

Suppose that
$$
f(z) = \sum\limits_{k=0}^{n}a_kz^k.
$$
Let
$$
f^{E}(z) = \sum\limits_{k=0}^{\lfloor n/2 \rfloor}a_{2k}z^k
\quad \text{and} \quad
f^{O}(z) = \sum\limits_{k=0}^{\lfloor (n-1)/2 \rfloor}a_{2k+1}z^k.
$$
Then, the following result is an equivalent form of Hermite-Biehler
Theorem.
\begin{thm}\label{thm+HS+tran}\emph{\cite[Theorem 6.3.4]{RS02}}
Let $f(z)=zf^{O}(z^2)+f^{E}(z^2)$ be a polynomial with real
coefficients. Suppose that $f^{E}(z)f^{O}(z)  \not\equiv 0$. Then
$f(z)$ is Hurwitz stable if and only if $f^{E}(z)$ and $f^{O}(z)$
have only real and non-positive zeros, and $f^{O}(z) \ll f^{E}(z)$.
\end{thm}

Thus, we have the following result.
\begin{prop}\label{prop+dec+peak}
Let $(a,b)$ be the (symmetric) $\mathcal{I}_n$-decomposition of
$\widetilde{M}_{n,2}(x)$. Then $T_{n,2}(x)$ is Hurwitz stable for
$n \in \mathbb{N}$ and $b \ll a$. In particular, $M_{n,2}(x)$ and
$\widetilde{M}_{n,2}(x)$ are alternatingly increasing for $n \in
\mathbb{N}$.
\end{prop}

\begin{proof}
 By (\ref{poly+peak}), for $r=2$, we get
\begin{equation}\label{Rec+peak}
 (1+x)T_{n,2}(x)=xM_{n,2}(x^2)+\widetilde{M}_{n,2}(x^2).
\end{equation} Moreover, we have
$$
T_{n+1,2}(x)=(2nx^2+x)T_{n,2}(x)+x(1-x^2)D_xT_{n,2}(x),
$$
where $T_{0,2}(x)=1$ and $T_{1,2}(x)=x$. This coincides with
\eqref{rec+gf+alt+run+Stir+Perm}. Thus Proposition \ref{prop+MW}
implies that $T_{n,2}(x)$ is Hurwitz stable. By Theorem
\ref{thm+HS+tran}, we have
$$
    \mathcal{I}_n(\widetilde{M}_{n,2}(x))
  = M_{n,2}(x) \ll \widetilde{M}_{n,2}(x).
$$
It is equivalent that $b \ll a$ by Theorem \ref{lem+Dec+ab}. And
thus, $\widetilde{M}_{n,2}(x)$ is alternatingly increasing for all
$n \in \mathbb{N}$.

Let $(\widetilde{a},\widetilde{b})$ be the $\mathcal{I}_{n-1}$-decomposition
of $M_{n,2}(x)$. Note that $M_{n,2}(x) = \mathcal{I}_n(\widetilde{M}_{n,2}(x))$
and the degree of $M_{n,2}(x)$ is $n-1$, then
$$
M_{n,2}(x) \ll \widetilde{M}_{n,2}(x) = \mathcal{I}_n(M_{n,2}(x))
           = x \mathcal{I}_{n-1}(M_{n,2}(x)).
$$
That is to say, $\mathcal{I}_{n-1}(M_{n,2}(x)) \ll M_{n,2}(x)$, i.e,
$\widetilde{b} \ll \widetilde{a}$. Thus, $M_{n,2}(x)$ is alternatingly
increasing for all $n \in \mathbb{N}$.

\end{proof}

%
%

\begin{re}
\em The alternatingly increasing property of $M_{n,2}(x)$ and
$\widetilde{M}_{n,2}(x)$ was also proved in \cite[Theorem
12]{MMYY20} in a different way. Here our Proposition
\ref{prop+dec+peak} gives a stronger result than the alternatingly
increasing property.
\end{re}

\end{document}